\documentclass[11pt]{article}
\usepackage[singlespacing]{setspace}
\usepackage{mathtools,amsmath, amssymb, amscd, amsthm, amsfonts}
\usepackage{mathtools}
\usepackage[mathscr]{euscript}
\usepackage{mathrsfs}
\usepackage{enumitem}
\usepackage{graphicx}
\usepackage{float}
\usepackage{tocloft}
\usepackage{psfrag}
\usepackage{caption}
\usepackage{subcaption}
\RequirePackage[hyphens]{url}
\RequirePackage{doi}
\usepackage{hyperref}
\hypersetup{
colorlinks = true,
urlcolor   = blue,
citecolor  = blue,
linkcolor  = blue,
}
\usepackage{xcolor}
\usepackage{scalerel,stackengine}
\usepackage[colorinlistoftodos]{todonotes}
\usepackage{authblk}
\usepackage{fancyhdr}

\newlength\FHoffset
\fancyheadoffset{\FHoffset}

\newlength\FHleft
\newlength\FHright

\renewcommand{\headrulewidth}{1.0pt} 
\newbox\FHline
\setbox\FHline=\hbox{\hsize=\paperwidth%
  \hspace*{\FHleft}%
  \rule{\dimexpr\headwidth-\FHleft-\FHright\relax}{\headrulewidth}\hspace*{\FHright}%
}

\title{\textbf{
Sharp Nonuniqueness in the Transport Equation with Sobolev Velocity Field
}
}
\author[]{\large \textbf{Elia Bru\`e}\footnote{Department of Decision Sciences, Universit\`a Bocconi, Milano, Italy. \textit{Email:}  \href{mailto:elia.brue@unibocconi.it}{elia.brue@unibocconi.it}. } \;
\textbf{\&} \;
\textbf{Maria Colombo}\footnote{EPFL, Station 8,
CH-1015 Lausanne,
Switzerland. \textit{Email:}  \href{mailto:maria.colombo@epfl.ch}{maria.colombo@epfl.ch}. } \; 
\textbf{\&} \; 
\textbf{Anuj Kumar}\footnote{Department of Mathematics, University of California Berkeley, CA 94720, USA. \textit{Email:}  \href{mailto:anujkumar@berkeley.edu}{anujkumar@berkeley.edu}. }}
\date{}

\newtheoremstyle{mystyle}
  {}
  {}
  {\itshape}
  {}
  {\bfseries}
  {.}
  { }
  {\thmname{#1}\thmnumber{ #2}\thmnote{ (#3)}}

\theoremstyle{mystyle}
\newtheorem{theorem}{Theorem}[section]
\newtheorem{proposition}[theorem]{Proposition}

\theoremstyle{definition}
\newtheorem{remark}{Remark}[section]

\newcommand\norm[1]{\left\lVert#1\right\rVert}

\newcommand{\bs}[1]{\boldsymbol{#1}}

\newcommand{\wt}[1]{\widetilde{#1}}
\newcommand{\ol}[1]{\overline{#1}}
\newcommand{\R}{\mathbb{R}}
\newcommand{\N}{\mathbb{N}}

\stackMath
\newcommand\reallywidecheck[1]{%
\savestack{\tmpbox}{\stretchto{%
  \scaleto{%
    \scalerel*[\widthof{\ensuremath{#1}}]{\kern-.6pt\bigwedge\kern-.6pt}%
    {\rule[-\textheight/2]{1ex}{\textheight}}
  }{\textheight}%
}{0.5ex}}%
\stackon[1pt]{#1}{\scalebox{-1}{\tmpbox}}%
}

\DeclareMathOperator\supp{supp}

\def\XXint#1#2#3{{\setbox0=\hbox{$#1{#2#3}{\int}$ }
\vcenter{\hbox{$#2#3$ }}\kern-.6\wd0}}

\numberwithin{equation}{section}

\begin{document}

\maketitle

\begin{abstract}
Given a divergence-free vector field ${\bf u} \in L^\infty_t W^{1,p}_x(\mathbb R^d)$ and a nonnegative initial datum $\rho_0 \in L^r$, the celebrated DiPerna--Lions theory established the uniqueness of the weak solution in the class of $L^\infty_t L^r_x$ densities for $\frac{1}{p} + \frac{1}{r} \leq 1$. This range was later improved in \cite{BrueColomboDeLellis21} to $\frac{1}{p} + \frac{d-1}{dr} \leq 1$.
We prove that this range is sharp by providing a counterexample to uniqueness when $\frac{1}{p} + \frac{d-1}{dr} > 1$.

To this end, we introduce a novel flow mechanism. It is not based on convex integration, which has provided a non-optimal result in this context, nor on purely self-similar techniques, but shares features of both, such as a local (discrete) self similar nature and an intermittent space-frequency localization. 

\end{abstract}









\section{Introduction}
\label{sec: intro}

This paper addresses the question of the uniqueness of the Cauchy problem associated with the {\it incompressible transport equation}:

\begin{equation}\label{eqn:CE}
\begin{cases}
    \partial_t \rho + {\bs u} \cdot \nabla \rho 
    = 0
    \\
    {\rm div} (\bs{u}) = 0
    \\
    \rho(0,x) = \rho_0(x) \, , \quad x \in \R^d \, .
\end{cases}\tag{TE}
\end{equation} 
Given an incompressible velocity field ${\bs u} \in L^1_t W^{1,p}_x$ and an initial datum $\rho_0 \in L^r$
, we consider 
solutions $\rho \in L^\infty_t L^r_x$ to \eqref{eqn:CE}, for integrability exponents $p, r \ge 1$.

\bigskip

The study of transport equation and ODEs with Sobolev velocity fields has a long history, beginning with the pioneering works of Di Perna and Lions \cite{DiPernaLions} and Ambrosio \cite{Ambrosio04},  
with applications to several PDE models of fluid dynamics and the theory of conservation laws (see the review \cite{AmbrCr}).

\bigskip

The problem of well-posedness of \eqref{eqn:CE} was first addressed by Di Perna and Lions in 1989. In their groundbreaking work \cite{DiPernaLions}, they demonstrated that the Cauchy problem \eqref{eqn:CE} has a unique solution $\rho \in L^\infty_t L^r_x$ within the following range of exponents:

\begin{equation}
\label{DPL}
\frac{1}{p} + \frac{1}{r} \leq 1 \, .
\end{equation}
Moreover, in this regime, solutions are {\it renormalized} and {\it Lagrangian}, i.e., they can be represented using the {\it Regular Lagrangian Flow} of ${\bs u}$. The latter is a generalized notion of flow suitable for the Sobolev framework. 
Similar conclusions can be obtained beyond the DiPerna Lions framework, under suitable structural assumptions, such as nearly incompressible \cite{BiaBo} or 2-dimensional autonomous vector fields \cite{ABC}.

\bigskip

On the negative side, the first counterexample to the uniqueness of solutions to \eqref{eqn:CE} has been constructed using the {\it method of convex integration}. The fundamental contributions \cite{MoSz2019AnnPDE,MoSa2019} have shown that nonuniqueness holds in the range of exponents:

\begin{equation}
\label{CI}
\frac{1}{p} + \frac{1}{r} > 1 + \frac{1}{d} \, ,
\end{equation}
with an incompressible vector field enjoying the additional integrability ${\bs u}\in L^\infty_tL^{\frac r{r-1}}_x$.

\bigskip

There is a gap between the well-posedness range \eqref{DPL} established by Di Perna and Lions and the nonuniqueness range \eqref{CI} obtained through convex integration.
The goal of this paper is  to 
determine the {\it sharp range of uniqueness} for {nonnegative solutions of the }transport equation.

\begin{theorem}
Let $d \geq 2$, $r \in [1, \infty)$ and $p \in [1, \infty)$ be such that
\begin{align}\label{cond:ip}
\frac{1}{p} + \frac{d-1}{dr} > 1\, .
\end{align}
Then there exists a compactly supported vector field $\bs{u} \in C([0, 1]; W^{1, p}\cap L^{\frac{r}{r-1}}(\mathbb{R}^d; \mathbb{R}^d))$ such that \eqref{eqn:CE} admits two distinct nonnegative solutions in the class $C([0, 1]; L^r(\mathbb{R}^d))$. 
\label{Intro: main thm}
\end{theorem}

\begin{remark}[Sharpness of Theorem \ref{Intro: main thm}] It was shown in \cite[Theorem 1.5]{BrueColomboDeLellis21} that, for $p,r\in [1,+\infty]$ satisfying condition \eqref{cond:ip} with the opposite inequality, namely
	\begin{equation}\label{eq:wellposednessrange}
	\frac{1}{p}+\frac{1}{r}<1+\frac{1}{d-1}\frac{r-1}{r},
	\end{equation}
and any vector field {$\bs{u}\in L^1([0,T],W^{1,r}(\mathbb{R}^d,\mathbb{R}^d))$} with bounded space divergence, solutions of \eqref{eqn:CE} are unique  among all nonnegative, weakly continuous in time densities $\rho\in L^{\infty}([0,T],L^p(\mathbb{R}^d))$ (resp., for $p=1$,  nonnegative weakly-star continuous densities $\rho\in L^{\infty}([0,T],\mathscr{M}(\mathbb{R}^d
	))$ with $\rho(0,\cdot)=\rho_0 \mathcal L ^d$).
\label{Intro: cited thm}
\end{remark}

 



\bigskip

Theorem \ref{Intro: main thm} definitively settles the well-posedness question within the Di Perna Lions class. However, we believe that the theoretical significance of this paper extends beyond the realm of linear transport theory. To our knowledge, this marks one of the rare instances in fluid dynamics where the uniqueness class has been precisely determined (compare with Remark \ref{rmk:Nonuniq}). Standard techniques, such as the Di Perna Lions commutator estimate and the convex integration approach, reveal limitations that do not allow reaching the sharp exponent.
Our 
new approach, crafted to overcome these challenges, reveals robust features with the potential for applications in other nonlinear problems.

To discuss more in detail the main ideas behind the proof of Theorem \ref{Intro: main thm} and provide some mathematical context, we briefly describe two techniques that have proven to be powerful in demonstrating nonuniqueness in nonlinear models such as the Euler and Navier-Stokes equations. Although our approach does not fit into either of these two techniques, it shares features with both.

\subsection{Convex integration}

This technique was introduced to fluid dynamics equations in the context of the {\it Onsager conjecture}, which has been completely proven in a remarkable sequence of results, including \cite{DeLellisSzekelyhidi09,DeLellisSzekelyhidi13,Isett2018Annals,BDLISZ15}. The convex integration technique, a nonlinear method, operates on the principle that the interaction of high-frequency functions can produce low-frequency terms, that can be used to correct error terms. This high-frequency interaction occurs through the perturbation of approximate solutions with suitably defined high-frequency perturbations.

\bigskip

Due to its flexibility, convex integration has found successful applications in various fluid dynamics models. Noteworthy is the groundbreaking work 
\cite{BuckmasterVicolAnnals}, introducing the concept of {\it spatial intermittency}. For a comprehensive account of applications of the convex integration technique, we refer the reader to \cite{DLSzsurvey,BVsurvey}, and references therein.

\bigskip

In the context of the linear continuity equation, the initial convex integration approaches were developed in \cite{MoSz2019AnnPDE,MoSa2019,MoSz2019CalcVar}, ultimately demonstrating the nonuniqueness of $L^\infty_t L^r_t$ solutions within the range \eqref{CI}. In this framework, the applicability of convex integration relies on the interaction between the velocity field ${\bs u}$ and the density $\rho$. Thus, the continuity equation is regarded as nonlinear in the pair $(\bs{u},\rho)$.

In the subsequent work \cite{BrueColomboDeLellis21}, convex integration has been used to show that there are vector fields in $\bs{u}\in W^{1,p}$, $p<d$ such that the trajectories
of the corresponding ODE are nonunique on a set of positive measure of initial data. This has been
later proved with different techniques in \cite{kumar2023nonuniqueness}.

In a slightly different direction, the temporal intermittency approach of \cite{MR4199851,MR4462623} (see \cite{MR4462623} for the Navier-Stokes equations) shows, in a context which is not directly comparable with the Di Perna-Lions theory due to low time integrability, that for vector fields in $L_t^1W_{\bs{x}}^{1,p}$ there are nonunique $L_t^q C_{\bs{x}}^k$ solutions for any $q<\infty$ and $k\in \N$.


\begin{remark}[Nonuniqueness vs Energy Conservation]\label{rmk:Nonuniq}

In the classical context of the Euler equations, the question of determining the regularity thresholds that dictate uniqueness remains open. Thanks to contributions on the Onsager conjecture, it is now known that the threshold for energy conservation in Hölder spaces is $\alpha_{EC} = 1/3$. In other words, $\alpha$-Hölder continuous solutions for $\alpha > \alpha_{EC}$ conserve energy, while for any $\alpha < \alpha_{EC}$ there exist dissipative solutions.
On the contrary, determining the value of the uniqueness threshold $\alpha_U$ remains a significant open problem (see the review \cite[Section 8, Problem 6]{MR4073888}). Currently, it is estimated that $1/3 \leq \alpha_U \leq 1$.
\end{remark}

\subsection{Instability of self-similar solutions and nonuniqueness}
\label{subsec:self-sim}

The study of self-similar solutions, namely families of exact solutions invariant under the intrinsic scaling invariance of the problem, is at the basis of our understanding of many nonlinear PDEs.
In the context of uniqueness problems, the {\it spectral stability} of self-similar solutions in {\it similarity variables} plays a central role. It is well-understood that unstable modes might generate unstable nonlinear manifolds leading to nonuniqueness of the Cauchy problem. 

\bigskip

{
In the framework of Leray solutions to the Navier-Stokes equations, this nonuniqueness program has been put forth by Jia Sverak and Guillod ~\cite{jiasverakselfsim,jiasverakillposed,guillodsverak}.
It is based on first constructing global self-similar solutions from every $-1$-homogeneous initial data, thereby extending beyond the small-data regime of \cite{kochtataru} within this particular class. Secondly, the conjecture put forth in~\cite{jiasverakillposed} suggests that nonuniqueness arises due to bifurcations within or from this class of solutions, reducing the problem to the study of spectral properties around the self-similar profiles. One of the spectral conditions was numerically verified in~\cite{guillodsverak} on certain axisymmetric examples with pure swirl initial data.
However, to date, there exists no rigorous proof of the existence of an unstable self-similar solution, despite the considerable numerical evidence.

}



 \bigskip

A related program for the two-dimensional Euler equations was initiated in~\cite{BressanAposteriori,BressanSelfSimilar}.
Vishik \cite{Vishik1,Vishik2} (see also~\cite{OurLectureNotes}) was the first to obtain a fully rigorous result in the context of these approaches, based on self-similarity and instability. He obtained the sharpness of the Yudovich class in the forced two-dimensional Euler equations, that is, non-uniqueness of solutions with vorticity in $L^p$, $p < +\infty$, with a force  in $L^1_t L^p_x$ in the vorticity equation. 

Building on a fundamental step in this result, namely the construction of an unstable vortex, 
in \cite{albritton2021non} it was proven the nonuniqueness of Leray solutions to the forced Navier-Stokes equations.


\subsection{Our approach}

In this paper, we introduce a novel method for obtaining nonunique solutions to the transport equation. This method exhibits both similarities and remarkable differences when compared to the two approaches outlined above.

\begin{itemize}
    \item[(i)] {\it Self-similar nature:}
    Our solution displays a local (discrete) self-similar nature. However, in contrast with the construction described in section \ref{subsec:self-sim}, the singular region of the vector field and the density form a {\it widespread Cantor-type set} instead of a single point.

    \item[(ii)] {\it {Time-intermittency,} space-frequency localization:}
   Our vector field is localized in both space and frequency at each time. This stands in stark contrast to the prevailing implementation of convex integration, where solutions exhibit a degree of homogeneity in space, uniformly at each time.
\end{itemize}



To describe more in detail some technical aspects of our construction, the natural starting point is the work \cite{kumar2023nonuniqueness}. The author builds a flow map $X_{\bs{v}}$ of a vector field
\begin{equation}
    \bs{v} \in C^\infty([0,1) \times \R^d) \cap C([0, 1]; W^{1, p}(\R^d;\R^d) \cap C^\alpha([0,1] \times \R^d)\, , \quad p<d
\end{equation}
which is uniquely defined for every $ \bs{x}\in \R^d$ and $0\le t <1$, and such that $X_{\bs{v}}$ maps a Cantor set $\mathcal{C}_\Phi$ of dimension $s<d$, to the uniform distribution at time $t=1$.


As a consequence, the vector field $\bs{u}$ obtained by time-reversing $\bs{v}$ admits two distinct measure-valued solutions of the continuity equation \ref{eqn:CE} with $\rho_0 = \mathscr{L}^d|_{[-1/2,1/2]^d}$.



\begin{figure}[h]
\centering
 \includegraphics[scale = 0.5]{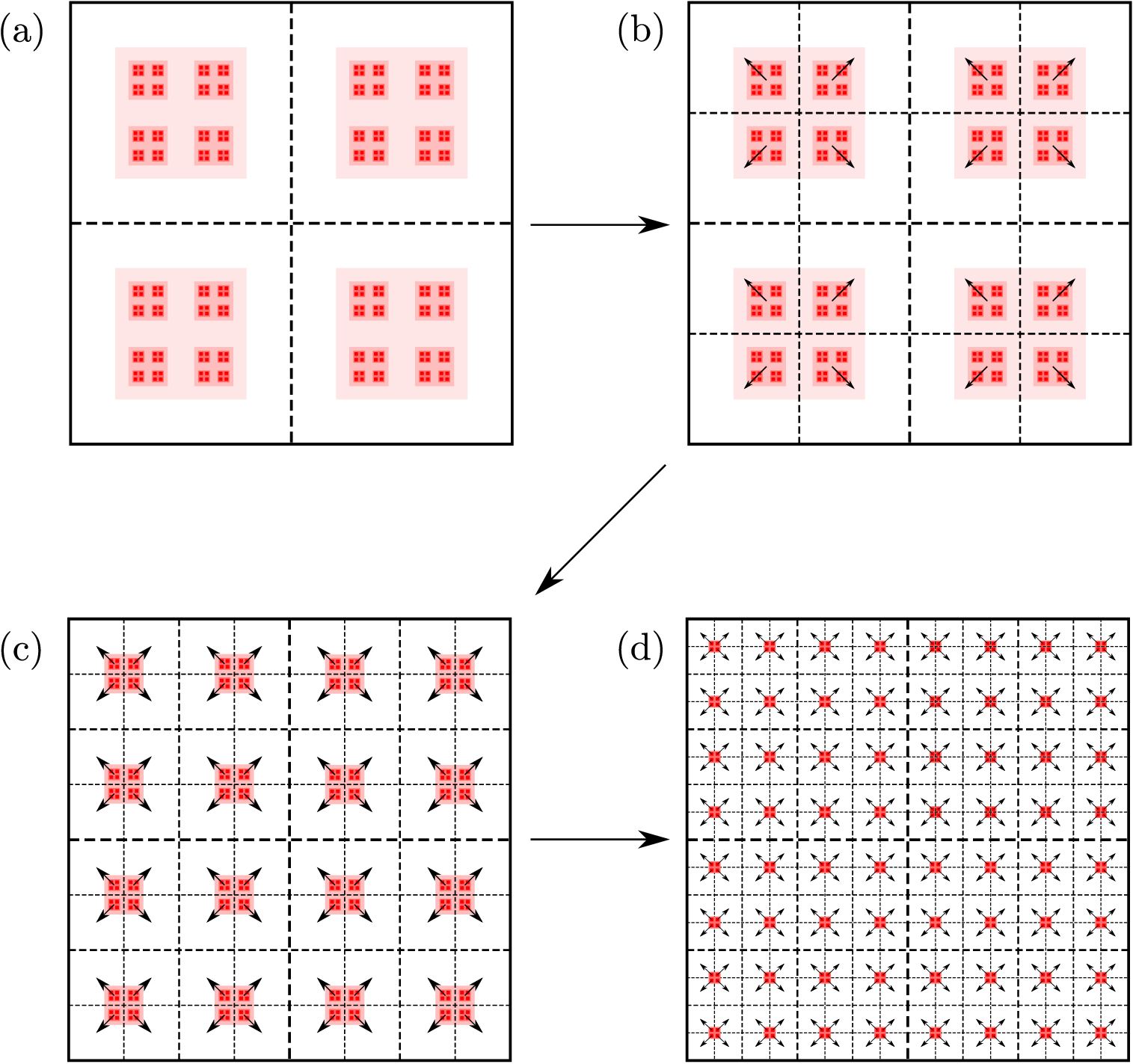}
 \caption{is adapted from \cite{kumar2022three}. The figure shows a few steps of how the vector field $\bs{v}$ distribute points in the Cantor set to the whole domain.}
 \label{fig: ODE paper}
 \end{figure}

\bigskip

The vector field $\bs{v}$ acts in infinite number of consecutive steps, where the $i$th step takes place over a time span of  $\frac{T_0}{2^{\beta i} }$ for some constants $T_0$, $0<\beta<1$.  In the $i$th step, the function of the vector field $\bs{v}$ is to translate the $i$th generation cubes in the Cantor set process so that their center is aligned with $i$th generation dyadic cube. The first four steps of this construction are shown in Figure \ref{fig: ODE paper}. The size of cubes translated at the $i$th stage by this vector field $\bs{v}$ is much smaller than the $i$th generation dyadic cubes and is the key reason for the control on the Sobolev norm of $\bs{v}$.

Finally, at the end of infinite steps the points in the Cantor set $\mathcal{C}_{\Phi}$ is distributed to the whole domain in finite time. We refer the reader to \cite[Section 3]{kumar2023nonuniqueness} for a complete overview of the construction of the vector field $\bs{v}$.



\bigskip

The construction above does not allow us to extend beyond measure-valued solutions, specifically to demonstrate nonuniqueness in the class of densities $\rho \in L^\infty_t L^r_{\bs{x}}$ for any $r \geq 1$. The density field $\rho(t,\cdot)$ becomes singular immediately after $t=0$.
To avoid the immediate concentration of mass and to align more closely with the uniqueness range found in \cite{BrueColomboDeLellis21}, we introduce a construction based on two novel ideas. The goal is to implement the features (i) and (ii) described above.

\begin{enumerate}[label = (\arabic*)]
    \item \emph{Interweaving the scaled copies of the vector field within itself:} We construct a vector field $\bs{v}$ using an infinite number of steps, following the approach in \cite{kumar2023nonuniqueness}. However, between the $i$th and $i+1$th steps, we  interweave spatially- and temporally-scaled copies of the vector field $\bs{v}$ itself. This first idea is detailed in Section \ref{sec: p=1} and enables us to establish nonuniqueness in the class of densities $\rho\in L^\infty_t L^1_x$, but no better.

This concrete implementation aligns with the self-similar Ansatz, featuring a widespread Cantor-type set, as mentioned earlier in this section.

    \item \emph{Asynchronous translation of cubes:}
     We introduce asynchronous motion of cubes in constructing the vector field $\bs{v}$. This imparts spatial heterogeneity to $\bs{v}$, distinguishing it from the constructions in \cite{kumar2023nonuniqueness} and the convex integration approach. This heterogeneity is crucial for achieving the sharp range.
\end{enumerate}

From a technical standpoint, the realization of these two ideas will rely on a fixed-point argument. The implementation of the idea of asynchronous translation of cubes will necessitate a non-standard setup, which is thoroughly described and motivated in Section \ref{sec: outline}.

\subsection{Motivation from the perspective of flow design}
Over the past decade, 
flow designs have played a role  in describing physical phenomena, such as turbulent flows, as well as in addressing mathematically motivated problems, such as constructing velocity fields to prove nonuniqueness in ODEs and PDEs arising in fluid mechanics. Some of the notable research directions and some of the references include:\begin{enumerate}[label = (\roman*)]
    \item Flow designs for optimal mixing rate for a given budget on the velocity field \cite{
    Bressanconjecture2003, CrDL,
    lindoeringmixing11,
    Thiffeaultmixingnorms12,
     Iyermixing14, YaoZlatos17,
AlbertiCrippaMazzucato19,
ElgindiZlatosuniversalmixer,
    coopermanmixing23}.
    \item Flow designs for optimal heat transport for a given budget on the velocity field \cite{hassanzadeh2014wall,tobasco2017optimal,  motoki2018optimal,
    kumar2022three,
    kumarthesis23,
    AlbenoptimalPRF23}. 
    \item Investigating flow configurations that result in enhanced and anomalous dissipation  \cite{drivas22anomdissp, Colomboanomalous23, Brueanomalous23, armstrong2023anomalous, elgindianomalous23}.
    \item 
    Producing counterexamples to existence or uniqueness of ``appropriate'' ODEs and PDEs solutions
    \cite{Aizeman78counter,
    DiPernaLions,
    depauw2003non, de2022smoothing,kumar2023nonuniqueness, pappalettera2023measure}. 
\end{enumerate}

{These problems are actively studied using theoretical, computational, and analysis techniques 
and flow design problems are expected to play a central role in our future understanding of many questions in fluid dynamics. For instance, at the theoretical level, many potential blow up phenomena such as for the Navier-Stokes equations still lack a candidate flow design.
From a physics standpoint they can explain natural phenomena observed in naturally occurring flows, such as the enhanced transport of nutrients to the ocean surface or the rapid angular momentum transport involved in the formation of stars. In terms of applications, flow designs can facilitate the devising of better control engineering strategies in fluid flows, such as the development of heat exchangers with improved efficiency.}


In this context, our paper presents 
a novel flow construction to the existing list of known mechanisms, with the particular features of spatial heterogeneity and of a singular set distributed in spacetime. 
These characteristics might suggest disorderliness within our constructions, but in reality, the flow designs considered in this paper exhibit a remarkable  organization such as  self-similar copies of themselves at various spatial and time scales
. As such our constructions provide interesting scenarios for applications in problems related to mixing and anomalous dissipation. Additionally, they have the potential to illuminate the inner workings of convex integration schemes, which can enable us to design improved variants of this technique.

\subsection{Structure of the paper}

Since we introduce a new technique for nonuniqueness, we present in Section~\ref{sec: p=1} the method in a simple scenario, namely we prove nonuniqueness of solutions in the class $L^\infty_t L^1_{\bs{x}}$. While the result is not new and it is a special case of Theorem \ref{Intro: main thm}, we present its proof for pedagogical reasons. Focusing on this simpler case allows us to introduce the new construction in a simple setting and helps us to explain better the associated subtle nuances with each new idea. We then clarify which additional  ideas are needed to improve it to the sharp result of Theorem \ref{Intro: main thm}. Section~\ref{sec: outline} contains the proof of the optimal result, which can be read independently of Section~\ref{sec: p=1} but is more involved.

\paragraph*{Acknowledgements}
EB would like to express gratitude for the financial support received from Bocconi University.
MC is supported by the Swiss State Secretariat for Education, Research and lnnovation (SERI) under contract number MB22.00034 through the project TENSE. 
A.K. gratefully acknowledges the support provided by a Dissertation–Year fellowship at UC Santa Cruz, which facilitated the completion of the initial phase of this work.



\section{Nonuniqueness of solutions in $L^\infty_t L_{\bs{x}}^1$}
\label{sec: p=1}

In this section, we prove the following nonuniqueness result. As announced in the introduction, this special case of Theorem~\ref{Intro: main thm} introduces the simplified framework of our construction and is presented for pedagogical reasons.

\begin{theorem}\label{thm:nonuniq L1}
    Let $d\ge 2$.
    For every $p<d$ there exists a compactly supported divergence-free  vector field
    \begin{equation}\label{u-reg}
        \bs{u} \in C([0, 1]; W^{1,p}(\mathbb{R}^d, \mathbb{R}^d)) \cap C^\alpha([0,1]\times \mathbb{R}^d) \, , \quad \text{for some $\alpha\in (0,1)$}\, ,
    \end{equation}
    such that there are at least two nonnegative solutions $\rho, \Tilde{\rho}\in L^\infty([0, 1]; L^1(\mathbb{R}^d))$ of the continuity equation starting from the same initial condition. 
\end{theorem}

\begin{figure}[h]
\centering
 \includegraphics[scale = 0.4]{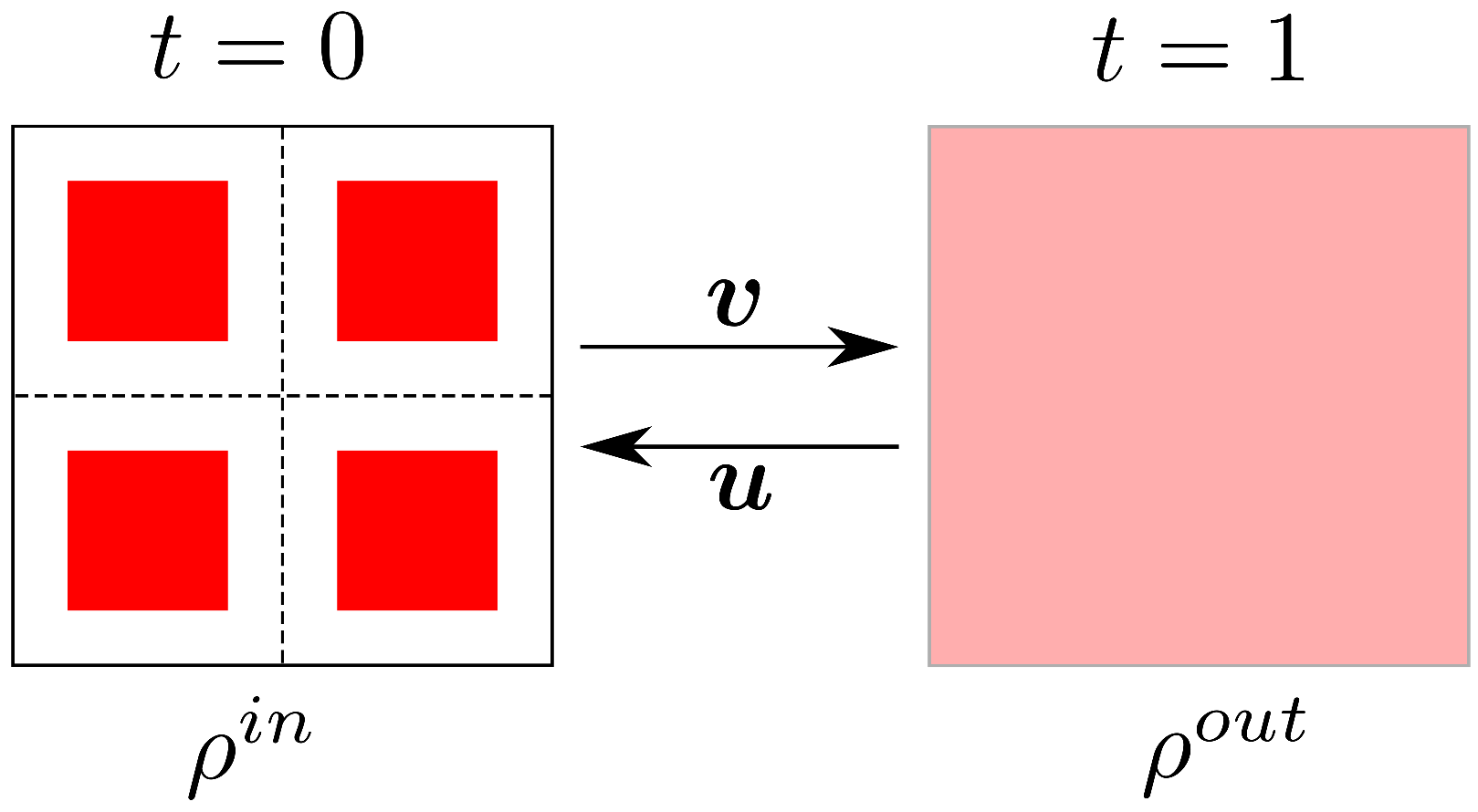}
 \caption{shows the initial and the final density of the vector fields $\bs{u}$ and $\bs{v}$. The intensity of the red color indicates the increase in the magnitude of the density. 
 }
 \label{fig: overall situation}
\end{figure}

\subsection{Overview of the construction}
\label{subsec: L1 overview}
We denote $Q(\bs{x}, \ell)$ to be the cube centered at $\bs{x}$ and has length $\ell$.
For convenience, we construct the vector field $\bs{u}$ as $\bs{u}(t, \cdot) \coloneqq -\bs{v}(1-t, \cdot)$, namely through a time reversal argument applied to a vector field $\bs{v} 
$ in the same regularity class \eqref{u-reg} for which there is a solution $\rho \in L^\infty([0, 1]; L^1(\mathbb{R}^d))$ with initial and final conditions given by
\begin{align}
\rho(0, \cdot) = \rho^{in} \qquad \text{and} \qquad \rho(1, \cdot) = \rho^{out} \equiv 1_{\ol{Q}(0, 1)}.
\end{align}
Here, $\rho^{in}$ concentrates on $2^d$ cubes of size $1/2^{(1 + \nu)}$ for some parameter $\nu > 0$, and the centers of these cubes align with the center of the dyadic cubes of the first generation, namely
\begin{align}
\rho^{in}(\bs{x}) \coloneqq
\begin{cases}
2^{\nu d} \qquad \text{if} \quad \bs{x} \in \bigcup\limits_{i_1, \dots , i_d \in \{-1, 1\}} \ol{Q}\left(\left(\frac{i_1}{4}, \dots \frac{i_d}{4}\right), \frac{1}{2^{1 + \nu}}\right), \\
0 \qquad \text{otherwise}.
\end{cases}
\end{align}

The vector field $\bs{v}$ stays supported inside $\ol{Q}(0, 1)=[-\frac 12,\frac12]^d$, the cube centered at the origin with the side of length $1$, at all times. The nonuniqueness stated in Theorem~\ref{thm:nonuniq L1} follows from the existence of such $\rho$: indeed, $\rho(1-t,x)$ solves the continuity equation with vector field $\bs{u}$ starting from $\rho^{out}$ and ending in $\rho^{in}$ (see Figure \ref{fig: overall situation}) and at the same time $\tilde \rho(t,x) \equiv \rho^{out}$ is another solution from the same initial datum.

\subsubsection{Time series}
Let $0<\beta<1$, we setup a time series as follows:
\begin{align}
\tau_i = 1 - \frac{1}{2^{\beta(i-1)}} \quad \text{for} \quad i \in \mathbb{N} \qquad \text{and} \qquad \tau_\infty = 1.
\label{def: time series L1}
\end{align}
Figure \ref{fig: available definition of w}{\color{blue}a} shows this time series. For convenience, we also define midpoints and the time intervals
\begin{align}
\tau^{mid}_i = \frac{\tau_i + \tau_{i+1}}{2}, \qquad \mathcal{E}_i \coloneqq [\tau_i, \tau_i^{mid}], \qquad \mathcal{O}_i \coloneqq (\tau_i^{mid}, \tau_{i+1}).
\end{align}

 \begin{figure}[H]
\centering
 \includegraphics[scale = 0.28]{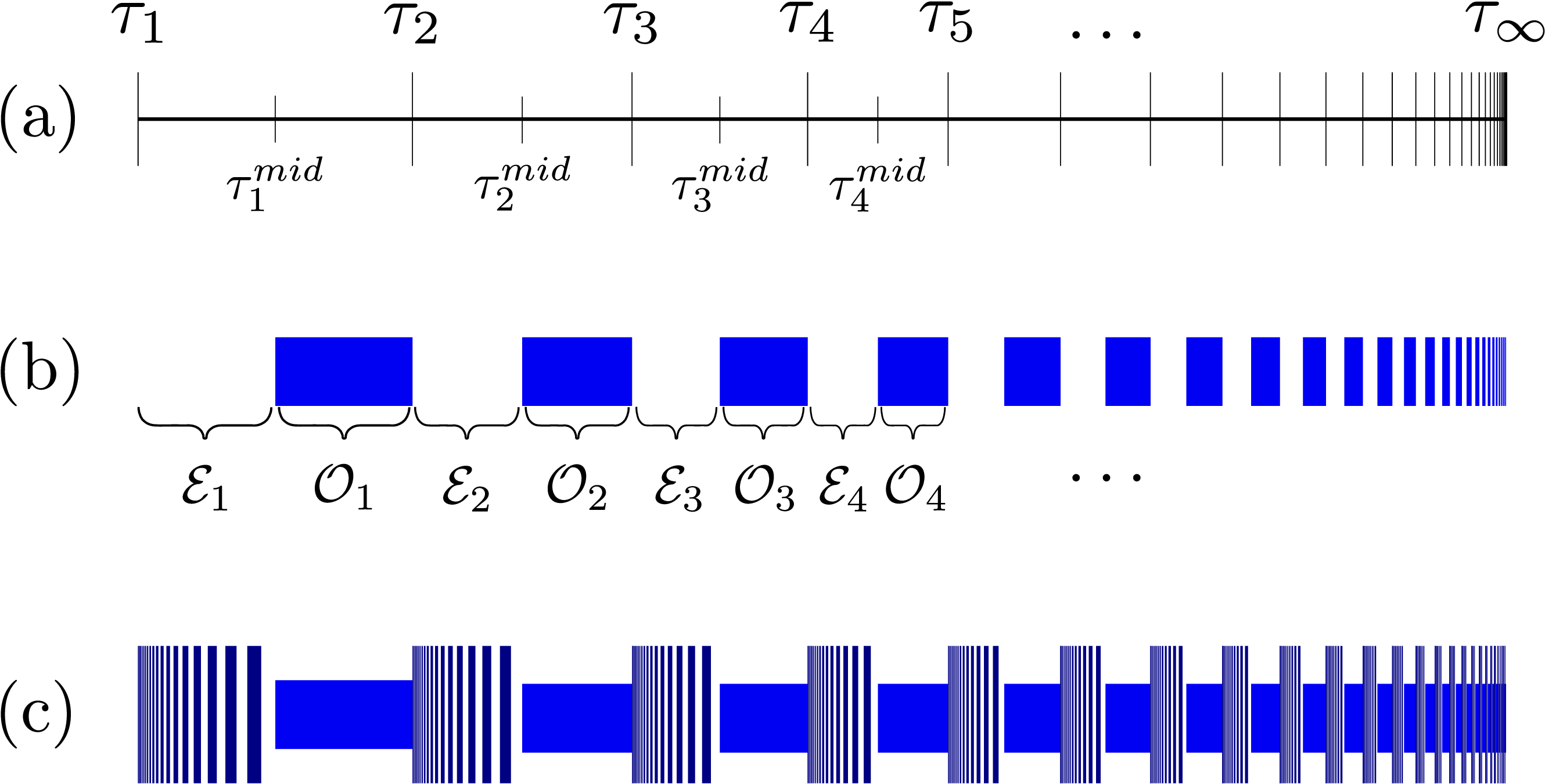}
 \caption{(a) shows the partition of the time interval $[0, 1]$ using the time series (\ref{def: time series L1}). Panel (b) shows the time intervals $\bigcup\limits_{i \in \N} \mathcal{O}_i$ using blue lines. The significance of $\bigcup\limits_{i \in \N} \mathcal{O}_i$ is that we define our fixed point relation $\mathcal{F}$ and $\mathcal{G}$ (see (\ref{def: L1 mapping F}) and (\ref{def: L1 mapping G})) explicitly on these intervals using the building block vector field from section \ref{sec: building block}. On the rest of the interval $[0, 1]$ (the white blanks in panel (b)) we use scaled copies of the vector field $\bs{v}$ (resp. the density field $\rho$) in our definition of $\mathcal{F}$ (resp. $\mathcal{G}$). Panel (c) illustrates how the explicit definitions are then progressively fed onto more time intervals in the fixed point iteration process. For the purpose of demonstration, we use taller and darker blue lines in panel (c).
}
 \label{fig: available definition of w}
\end{figure}

\subsubsection{The action of $\bs{v}$}
We now elaborate on the action of the vector field $\bs{v}$ which is also displayed in figure \ref{fig: density evolution}. 






At time $t = \tau_i$ (to fix ideas, consider $i=1$ and $\tau_1=0$), the density $\rho$ concentrates on $2^{id}$ cubes of size $\frac{1}{2^{(1+\nu)i}}$ and the center of these cubes align with the center of the $i$th generation dyadic cubes. 
From time $t = \tau_i$ to $t = \tau^{mid}_i$ these cubes break into $2^{(i+1)d}$ cubes and the magnitude of the density goes up such that the $L^1$ norm remains preserved.
The key observation, which we will delve into further in the upcoming paragraphs, is that the construction of the vector field $\bs{v}$ in the time interval $[\tau_i, \tau_i^{mid}]$ depends on the spatially-temporally-scaled copies of $\bs{v}$ itself.
. In the time interval $\mathcal{O}_i = (\tau^{mid}_i, \tau_{i+1})$, the vector field $\bs{v}$ is defined explicitly in terms of a building block vector field (see (\ref{def: L1 mapping F})) and spreads these cubes so that their center now align with the centers of the $i+1$-th generation dyadic cubes.


\bigskip

With the above description, the definition of the vector field $\bs{v}$ on the time intervals $\mathcal{E}_i = [\tau_i, \tau_i^{mid}]$ (where the cubes break into smaller cubes) is still missing. The most crucial observation we make that going from time $t = \tau_i$ to time $t = \tau_i^{mid}$ consists of $2^{id}$ smaller versions of the original problem of moving $\rho^{in}$ to $\rho^{out}$. This observation is the key that allowed us to go beyond the class of measure solutions and to prove nonuniqueness in $L^\infty_t L^1_{\bs{x}}$. Figure \ref{fig: w i tilde w i overbar} illustrates this point on the time interval $\mathcal{E}_1 = [\tau_1, \tau_1^{mid}]$. Our observation therefore inspires the construction of the vector field $\bs{v}$ and the density field $\rho$ using a fixed point argument, where the definition of the fixed point relation is explicit on the time intervals $\mathcal{O}_i = (\tau_{i}^{mid}, \tau_{i+1})$ (shown using blue lines in figure \ref{fig: available definition of w}{\color{blue}b}). Figure \ref{fig: available definition of w}{\color{blue}c}  shows how this explicit definition would be fed on more and more of the time interval $[0, 1]$ in the fixed point iteration process and eventually covering the time interval $[0, 1]$ almost everywhere.

\subsection{
The centers of the $\eta$-th generation dyadic cubes}
\label{subsec: L1 distinguished points}

Given $\eta \in \mathbb{N}$, we divide the $[-\frac{1}{2}, \frac{1}{2}]^d$ into $2^{\eta d}$ cubes with sides of length $2^{-\eta}$. We denote by $\bs{c}^\eta_k \in [-\frac{1}{2}, \frac{1}{2}]^d$, $k \in \{1, 2, \dots , 2^{\eta d}\}$, the centers of the cubes of the $\eta$-th subdivision, namely
\begin{align}
\bs{c}^\eta_k = \left(\frac{1}{2^\eta} \left[j_{\eta, 1}^k - \frac{2^\eta - 1}{2}\right], \dots \frac{1}{2^\eta} \left[j_{\eta, d}^k - \frac{2^\eta - 1}{2}\right]\right) \, ,
\label{def: c eta k}
\end{align}
where $(j_{\eta, 1}^k, \dots ,j_{\eta, d}^k) \in \{0, \dots ,2^\eta - 1\}^d$ is the unique $d$-tuple  such that $\sum_{l = 1}^d 2^{\eta (l-1)} j_{\eta, l}^k = k-1$.

\subsection{Construction of the vector field $\bs{v}$}
\label{subsec: L^1 vel construction}

As explained in the overview section \ref{subsec: L1 overview}, our construction of the vector field $\bs{v}$ is based on a fixed point iteration applied to a mapping $\mathcal{F}$. 
The goal of this section is to present the functional analytic set-up. Given $\alpha\in (0,1)$ and $p\ge 1$, we define
\begin{align}
\mathcal{X} \coloneqq \big\{ \bs{w} \in  C([0, 1]; \dot W^{1, p}(\mathbb{R}^d, \mathbb{R}^d)) \cap C^\alpha([0, 1] \times \mathbb{R}^d, \mathbb{R}^d)  \; \big| \qquad \qquad \qquad \qquad \qquad \qquad \qquad \qquad \nonumber \\
\nabla \cdot \bs{w} \equiv 0, \; \bs{w}(0, \cdot) = \bs{w}(1, \cdot) \equiv \bs{0},  \; \supp_{\bs{x}} \bs{w}(t, \cdot) \subseteq \ol{Q}(0, 1)
\big\},
\end{align}
equipped with the norm
\begin{align}\label{eq.norms1}
\norm{\bs{w}}_{\mathcal{X}}
:=
\norm{\bs{w}}_{C_t \dot W^{1, p}_{\bs{x}}} + \norm{\bs{w}}_{C^\alpha_{t, \bs{x}}} \, ,
\end{align}

\begin{figure}[H]
\centering
 \includegraphics[scale = 0.4]{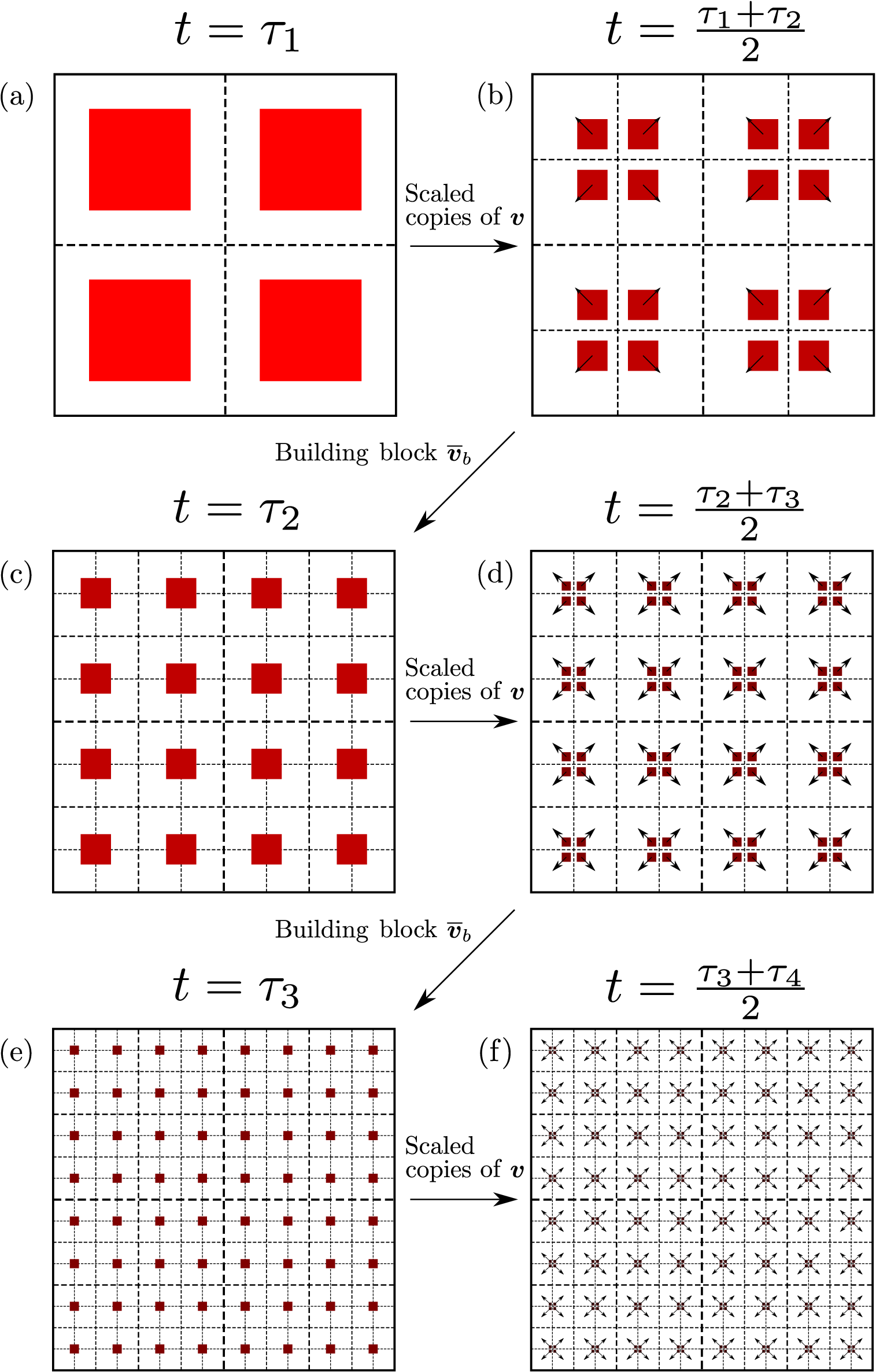}
 \caption{shows the evolution of the density $\rho$ (under the flow of vector field $\bs{v}$) at discrete times: $\tau_1$, $\tau_1^{mid} = \frac{\tau_1 + \tau_2}{2}$, $\tau_2$, $\tau_2^{mid}  = \frac{\tau_2 + \tau_3}{2}$, $\tau_3$ and $\tau_3^{mid} = \frac{\tau_3 + \tau_4}{2}$. The increase in the magnitude of the density is shown through the darkening of the red color.  The time interval $\mathcal{E}_i = [\tau_i, \tau_i^{mid}]$ is where the cubes break into smaller cubes via a scaled copy of the vector field $\bs{v}$ itself as illustrated in figure \ref{fig: w i tilde w i overbar}. On the time intervals $\mathcal{O}_i = (\tau_i^{mid}, \tau_{i+1})$ the cubes spread using the building block vector field $\ol{\bs{v}}_b$ from section~\ref{sec: building block}.}
 \label{fig: density evolution}
\end{figure}

 \begin{figure}[h]
\centering
 \includegraphics[scale = 0.35]{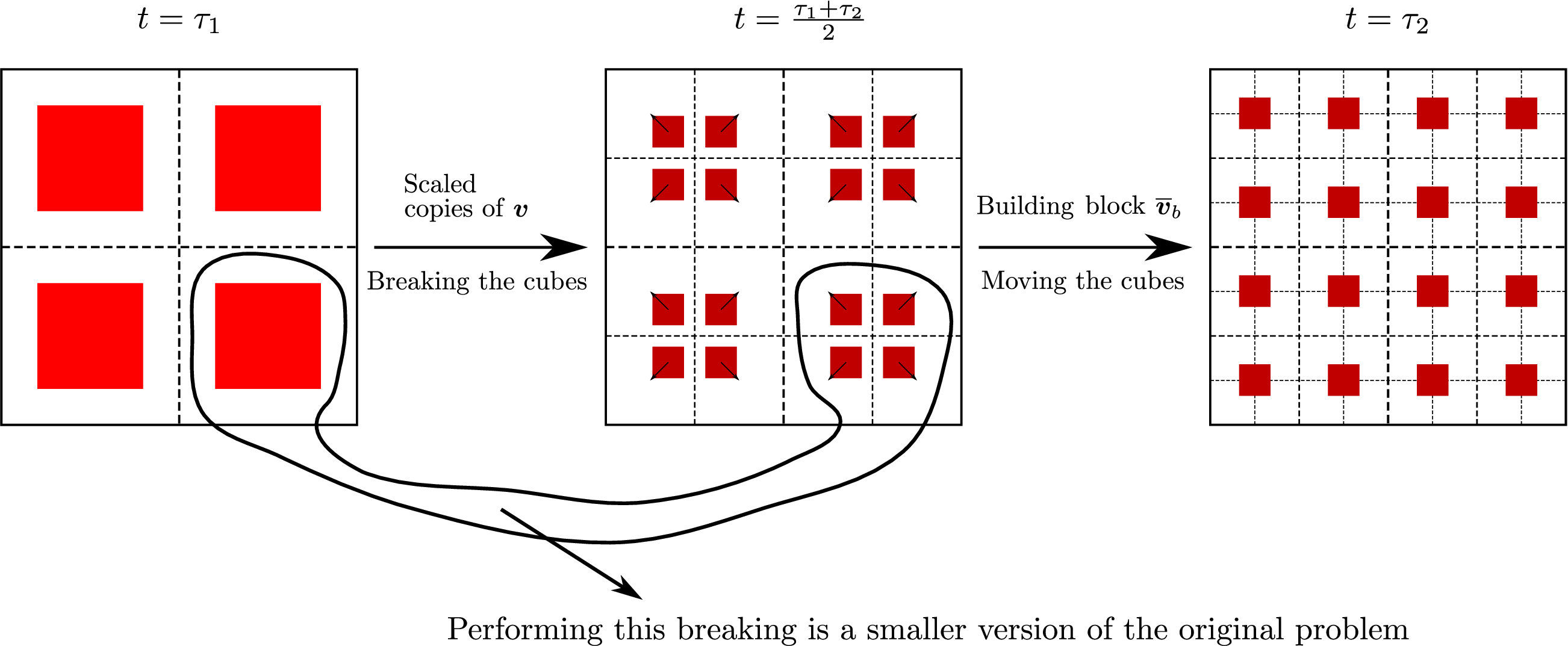}
 \caption{illustrates that going from $t = \tau_1$ to $t = \frac{\tau_1 + \tau_2}{2}$ consists of four scaled copies of the original problem (i.e., going from $\rho^{in}$ to $\rho^{out}$) but reversed in time. Going from $t = \frac{\tau_1 + \tau_2}{2}$ to $t = \tau_2$ requires translating the cubes which we do using the building block vector field $\ol{\bs{v}}_b$ from section \ref{sec: building block}.}
 \label{fig: w i tilde w i overbar}
\end{figure}

\begin{figure}[h]
\centering
 \includegraphics[scale = 1.3]{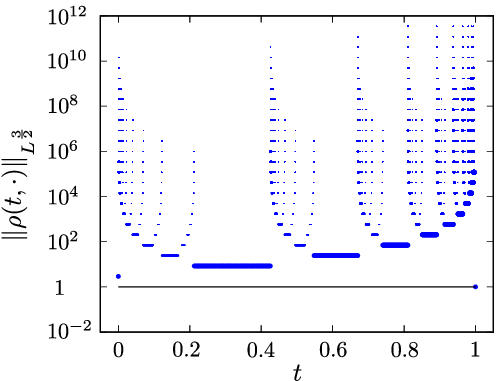}
 \caption{ shows a plot of the $\norm{\rho(t, \cdot)}_{L^{\frac{3}{2}}}$ as a function of time $t$. This figure corresponds to $d = 2$, $\beta = 0.8$ and $\nu = 2.3$. The $L^{\frac{3}{2}}$ norm of the density field is not bounded and blows up at several instances in time. We also added a plot of $\norm{\rho(t, \cdot)}_{L^{1}}$ for reference. The $L^1$ norms stays constant (in particular bounded) in time which is
 to be expected.}
 \label{fig: plot Lr norm density L1 consturction}
\end{figure}

where we denoted
\begin{align}\label{eq:norms}
   &\norm{f}_{\dot W^{1,p}} := \norm{\nabla f }_{ L^p}
   \\
   &\norm{f}_{C^\alpha} := \sup_{\bs{x}\neq \bs{y}} \frac{|f(\bs{x}) - f(\bs{y})|}{|\bs{x} - \bs{y}|^\alpha}\, .
\end{align}
\begin{remark}\label{rmk:L1 int}
    Each $\bs{w}\in \mathcal{X}$ belongs to $L^\infty_t L^p_{\bs{x}}$ by Sobolev embedding and compactness of the support.
\end{remark}

Given $\bs{w} \in \mathcal{X}$, we define $\mathcal{F}(\bs{w})$ separately on each time intervals $\mathcal{E}_i = [\tau_i, \tau_i^{mid}]$, $\mathcal{O}_i = (\tau_i^{mid}, \tau_{i+1})$.


\begin{subequations}
\begin{enumerate}[label = (\roman*)]
    \item When $t \in \mathcal{E}_i$, we define
    \begin{align}
        \mathcal{F}(\bs{w})(t, \bs{x}) \coloneqq \sum_{k \in \{1, \dots , 2^{id}\}} - \frac{1}{2^{(1+\nu)i}} \frac{1}{\tau^{mid}_i - \tau_i} \bs{w}\left(\frac{\tau^{mid}_i - t}{\tau^{mid}_i - \tau_i}, 2^{(1+\nu)i} (\bs{x} - \bs{c}_k^i)\right).
    \end{align}
        \item When $t \in \mathcal{O}_i$, we define $ \mathcal{F}(\bs{w})(t, \bs{x}) \coloneqq \ol{\bs{v}}_{b, i}(t,\bs{x})$,
    where
    \begin{align}
      \ol{\bs{v}}_{b, i}(t, \bs{x}) \coloneqq  \sum_{k \in \{1, \dots ,2^{id}\}}  \frac{1}{2^{i}} \frac{1}{\tau_{i+1} - \tau^{mid}_i} \; \ol{\bs{v}}_b\left(\frac{t - \tau^{mid}_i}{\tau_{i+1} - \tau^{mid}_i}, 2^{i} (\bs{x} - \bs{c}_k^i); 1, \frac{1}{2^{\nu i}}, \frac{1}{2^{\nu(i+1)}}\right)
    \end{align}
    and $\ol{\bs{v}}_b$ is the building block defined in Section~\ref{sec: building block}.
    \item Finally, when $t = 1$, we define
        $\mathcal{F}(\bs{w})(t, \cdot) \equiv \bs{0}.$
\end{enumerate}
\label{def: L1 mapping F}
\end{subequations}

\begin{proposition}
Let $d \geq 2$, $0 < \beta < 1$ and $\nu > \nu_0 \coloneqq 1 - \log_2(2^\beta - 1)$. If
\begin{align}\label{cond1}
p < \frac{\nu d}{\nu + \beta}, \qquad \text{and} \qquad \alpha < \frac{1 - \beta}{1 + \nu}\, ,
\end{align}
then $\mathcal{F}: \mathcal{X} \to \mathcal{X}$ is a contraction.
\end{proposition}
\begin{proof}
The proof of this proposition shows that $\mathcal{F}$ maps $\mathcal{X}$ to $\mathcal{X}$, and that $\mathcal{F}$ is a contraction.



Let $\bs{w} \in \mathcal{X}$, it is clear from the definition of $\mathcal{F}$ in (\ref{def: L1 mapping F}), points $\bs{c}^\eta_k$ from (\ref{def: c eta k}), and the building block vector field $\ol{\bs{v}}_b$ from Proposition \ref{prop: building block vector field} that 
\begin{align}
\nabla \cdot \mathcal{F}(\bs{w})(t, \cdot) = 0, \quad \mathcal{F}(\bs{w})(0, \cdot) = \mathcal{F}(\bs{w})(1, \cdot) \equiv \bs{0}, \quad \text{and} \quad \supp_{\bs{x}} \mathcal{F}(\bs{w})(t, \cdot) \subseteq \left[-\frac{1}{2}, \frac{1}{2}\right]^d.
\end{align}

We estimate the Sobolev norm $\dot W^{1, p}$ of $\mathcal{F}(\bs{w})$.
From the definition of $\mathcal{F}$ in (\ref{def: L1 mapping F}), we obtain that 
\begin{subequations}
     when $t \in \mathcal{E}_i$, we have
\begin{align}
\norm{\mathcal{F}(\bs{w}) - \mathcal{F}(\wt{\bs{w}})}_{C_t \dot W_{\bs{x}}^{1, p}}  = \norm{\mathcal{F}(\bs{w}-\wt{\bs{w}})(t, \cdot)}_{\dot W^{1, p}} \leq   \frac{1}{\tau^{mid}_i - \tau_i} \left(\frac{1}{2^{\nu d i}}\right)^{\frac{1}{p}}   \norm{\bs{w}-\wt{\bs{w}}}_{C_t \dot W^{1, p}_{\bs{x}}}   \, .
\end{align}

 When $t \in \mathcal{O}_i$, $\mathcal{F}(\bs{w})(t,\cdot) - \mathcal{F}(\wt{\bs{w}})(t, \cdot) \equiv 0$ and by scaling properties of the norm and Proposition~\ref{prop: building block vector field}, we have
\begin{align}
\norm{\mathcal{F}(\bs{w})(t, \cdot)}_{\dot W^{1, p}} \leq C(\nu, d, p)   \frac{2^{\nu i \left( {1-\frac{d}{p}}  \right)}}{\tau_{i+1} - \tau^{mid}_i}   \, .
\end{align}

\label{L1 W1p estimate F}
\end{subequations}

We estimate the H\"older norm $C^\alpha_{t, \bs{x}}$ of $\mathcal{F}(\bs{w})$.
\begin{subequations}
By explicit computations of rescalings of norms in the definition of $\mathcal{F}(\bs{w})$, we have 
\begin{align}
\norm{\mathcal{F}(\bs{w}) - \mathcal{F}(\wt{\bs{w}})}_{C^\alpha(\mathcal{E}_i \times \R^d)} 
&= \norm{\mathcal{F}(\bs{w}-\wt{\bs{w}})}_{C^\alpha(\mathcal{E}_i \times \R^d)}\nonumber
\\&\leq 
\frac{1}{2^{(1+\nu)i}}  \frac{1}{\tau^{mid}_i - \tau_i} \max\left\{\frac{1}{\tau^{mid}_i - \tau_i}, 2^{(1+\nu)i}\right\}^\alpha   \norm{\bs{w}-\wt{\bs{w}}}_{C^\alpha_{t, \bs{x}}} 
\nonumber
\\
&\leq C(\beta) 2^{-(1-\beta-\alpha)i} \norm{\bs{w}-\wt{\bs{w}}}_{C^\alpha_{t, \bs{x}}} \label{eqn:contrF1}
\end{align}
 The standard interpolation inequality and the estimates on $\ol{\bs{v}}_{b}$ in Proposition~\ref{prop: building block vector field} 
\begin{align}
\norm{\mathcal{F}(\bs{w})}_{C^\alpha(\mathcal{O}_i \times \mathbb{R}^d)} 
&\leq 2 \norm{{\mathcal{F}(\bs{w})}}_{L^\infty_{t, \bs{x}}}^{1 - \alpha} \norm{\nabla_{(t, \bs{x})}{\mathcal{F}(\bs{w})}}_{L^\infty_{t, \bs{x}}}^\alpha \, ,
\nonumber
\\ &\leq C 
\frac{1}{2^{i}} \frac{1}{\tau_{i+1} - \tau^{mid}_i} \max\left\{\frac{1}{\tau_{i+1} - \tau^{mid}_i}, 2^{i}\right\}^\alpha   2^{\nu \alpha i}  
\nonumber
\\& \leq C(\beta) 2^{-(1-\beta-\alpha - \nu \alpha)i}
\label{L1 Holder bound interval O}
\end{align}
\label{L1 C1alpha estimate F}
\end{subequations}

From the continuity of  $\bs{w}$ and $\ol{\bs{v}}_b$, it is clear that $\mathcal{F}(\bs{w})$ is continuous in the intervals $\mathcal{E}_i$, $\mathcal{O}_i$ for every $i\in \mathbb{N}$. 
 $\mathcal{F}(\bs{w})$ is also continuous at the interfaces $\tau_i$, $\tau_i^{mid}$ and $\tau_\infty$, where it is identically zero, thanks to the boundary conditions $\bs{w}(0, \cdot) = \bs{w}(1, \cdot) \equiv \bs{0}$ and the fact that $\supp_t \ol{\bs{v}}_b \subseteq \left[\frac{1}{3}, \frac{2}{3}\right]$. 
\bigskip

{\bf Step 1: $\mathcal{F}$  maps  $\mathcal{X}$ to $\mathcal{X}$.}
At this point, we can combine (\ref{L1 W1p estimate F}{\color{blue}{a-b}})  with $\wt{\bs{w}}=0$  to deduce
\begin{align}
\norm{\mathcal{F}(\bs{w})}_{C_t \dot W_{\bs{x}}^{1, p}} \leq C(\beta, \nu, d, p)  \sup_{i \in \mathbb{N}} \max\left\{ 2^{i \left(\beta - \frac{\nu d}{p}\right)}, \; 2^{\beta i} \, 2^{\nu \left( {1-\frac{d}{p}}  \right) i} \right\} \max\left\{1, \norm{\bs{w}}_{C_t \dot W^{1, p}_{\bs{x}}}\right\}.
\end{align}
Therefore, to ensure $\mathcal{F}(\bs{w}) \in C([0, 1]; \dot W^{1,p}(\mathbb{R}^d, \mathbb{R}^d))$, we impose the first requirement in \eqref{cond1}, namely
$
p < \frac{\nu d}{\beta + \nu}
\label{L1 condition on p F map X to X}
$. 
Similarly, we can combine (\ref{L1 C1alpha estimate F}{\color{blue}{a-b}})
to find
\begin{align}
\norm{\mathcal{F}(\bs{w})}_{C^\alpha_{t, \bs{x}}} 
\leq C(\beta)  
\sup_{i \in \mathbb{N}} 2^{-(1-\beta-\alpha-\nu \alpha)i} \max\{1, \norm{\bs{w}}_{C^\alpha_{t, \bs{x}}} \} \, ,
\end{align}
then to ensure $F(\bs{v}) \in C^\alpha_{t, \bs{x}}$, we only need  the second requirement in \eqref{cond1}, $\alpha < \frac{1 - \beta}{1+\nu}$.

\bigskip

{\bf Step 2: $\mathcal{F}$ is a contraction.}
Assuming (\ref{L1 condition on p F map X to X}), the constant on the right-hand side of (\ref{L1 W1p estimate F}{\color{blue}{a}}) 
is the largest for $i=1$, and it is smaller than $1$ when
\begin{align}
p < \frac{\nu d}{1 + \beta - \log_2(2^\beta - 1)} \, .
\end{align}
The latter condition is automatically satisfied, once we have (\ref{L1 condition on p F map X to X}) and
\begin{align}
\nu > 1 - \log_2(2^\beta - 1).
\label{F Holder cond. 2 nu}
\end{align}
Under \eqref{F Holder cond. 2 nu}, to make the constant on the right-hand side of (\ref{L1 C1alpha estimate F}{\color{blue}{a}}) 
less than $1$, we need
\begin{align}
\alpha < 1 - \frac{\beta - \log_2(2^\beta - 1)}{1 + \nu} \, .
\label{F Holder cond. 1 alpha}
\end{align}
However, the condition (\ref{F Holder cond. 1 alpha}) is met if \eqref{cond1} holds.
\end{proof}

\subsection{Construction of the density field $\rho$}
\label{subsec: L^1 density construction}
The construction of the density field is similar to that of the vector field from the previous subsection. We will define a mapping $\mathcal{G}$ and will look for fixed points. We begin by building the functional setup, which is this time not due only to the natural spaces in which $\rho$ lies, but also related to have a proper contraction (therefore we cannot work in $L^\infty([0,1]; L^1(\R^d))$) and to guarantee that the continuity equation is solved by the fixed point. For $s, q<\infty$ we define
\begin{align}
\mathcal{Y} \coloneqq \big\{ \varrho \in   L^s([0, 1]; L^1(\mathbb{R}^d)) \; \big| 
\; \rho \geq 0 \mbox{ a.e.},\; \supp_{\bs{x}} \varrho(t, \cdot) \subseteq 
\ol{Q}(0, 1)
\big\} \, ,
\end{align}
\begin{align}
\mathcal{Z} \coloneqq \big\{ \varrho \in   \text{Lip}([0, 1]; W^{-1, q}(\mathbb{R}^d)) \; \big| 
\varrho(0, \cdot) = \rho^{in}, \; \varrho(1, \cdot) = \rho^{out}, \; \supp_{\bs{x}} \varrho(t, \cdot) \subseteq \ol{Q}(0, 1)
\big\}\, .
\end{align}
Recall that $f\in L^1_{\rm loc}(\R^d)$ belongs to $W^{-1, q}(\mathbb{R}^d)$ if 
\begin{align}\label{eq:norms2}
\norm{f}_{W^{-1, q}} \coloneqq   \sup\limits_{\substack{\phi \in W^{1, q^\prime} \\ \norm{\nabla \phi}_{L^{q^\prime}} \leq 1}} \int_{\mathbb{R}^d} f \phi \; {\rm d} x < \infty \, .
\end{align}
We endow $\text{Lip}([0, 1]; W^{-1, q}(\mathbb{R}^d))$ with the homogeneous norm
\begin{equation}\label{eq:norms3}
    \| \varrho \|_{\text{Lip}([0, 1]; W^{-1, q}(\mathbb{R}^d))}
    := 
    \sup_{t,s\in [0,1], \\ t\neq s} \frac{1}{|t-s|} \| \varrho(t,\cdot) - \varrho(s,\cdot) \|_{W^{-1, q}}
\end{equation}



\bigskip

We will look for contraction in $\mathcal{X} = \mathcal{Y} \cap \mathcal{Z}$ endowed with the norm
\begin{align}\norm{\varrho
}_{\mathcal{X}} = 
\norm{\varrho}_{L^s_tL^1_{\bs{x}}} + \norm{\varrho}_{\text{Lip}_t W^{-1,q}_{\bs{x}}}\, .
\end{align}
We work with the norm $L^s_t L^1_{\bs{x}}$ for some $s < \infty$, even though our ultimate goal is to create a solution in $L^\infty_t L^1_{\bs{x}}$. This choice is made because the conservation of mass does not allow contraction in $L^\infty_t L^1_{\bs{x}}$. However, once we find a fixed point in $L^s_t L^1_{\bs{x}}$, we will upgrade it to $L^\infty_t L^1_{\bs{x}}$ using an iteration procedure.

Now, we define the mapping $\mathcal{G}$ separately on the time intervals $\mathcal{E}_i$ and $\mathcal{O}_i$:

\begin{subequations}
\begin{enumerate}[label = (\roman*)]
    \item When $t \in \mathcal{E}_i$, we define
    \begin{align}
        \mathcal{G}(\varrho)(t, \bs{x}) \coloneqq \sum_{k \in \{1, \dots , 2^{id}\}} 2^{\nu d i} \varrho\left(\frac{\tau^{mid}_i - t}{\tau^{mid}_i - \tau_i}, 2^{(1+\nu)i} (\bs{x} - \bs{c}_k^i)\right).
    \end{align}
        \item When $t \in \mathcal{O}_i$, we define $\mathcal{G}(\varrho)(t, \bs{x}) \coloneqq \ol{\rho}_{b, i}(t, \bs{x})$ where
    \begin{align}
       \ol{\rho}_{b, i}(t, \bs{x}) \coloneqq \sum_{k \in \{1, \dots ,2^{id}\}} 2^{\nu d (i+1)} \; \ol{\rho}_b\left(\frac{t - \tau^{mid}_i}{\tau_{i+1} - \tau^{mid}_i}, 2^{i} (\bs{x} - \bs{c}_k^i); 1, \frac{1}{2^{\nu i}}, \frac{1}{2^{\nu(i+1)}}\right).
    \end{align}
    \item Finally, we set $\mathcal{G}(\varrho)(1, \cdot) \equiv \rho^{out}$.
\end{enumerate}
\label{def: L1 mapping G}
\end{subequations}

\begin{proposition}
Let $d\ge 2$, $1 < s,q<\infty$. If $\nu \ge 2$ and
\begin{align}\label{eq3}
\frac{\nu d}{1 - \beta} < \frac{q}{q-1} \, ,
\end{align}
then $\mathcal{G}: \mathcal{X}\to \mathcal{X}$ is a contraction.
\end{proposition}

\begin{proof}
    
{We estimate the $L^s_t L^1_{\bs{x}}$ norm}.
\begin{subequations}
    When $t \in \mathcal{E}_i$, we get
    \begin{align}
        \int_{\mathbb{R}^d}|\mathcal{G}(\varrho)|(t, \cdot)\; {\rm d}\bs{x} = \int_{\mathbb{R}^d} |\varrho|\left(\frac{\tau^{mid}_i - t}{\tau^{mid}_i - \tau_i}, \cdot \right) \; {\rm d}\bs{x}.
    \end{align}
When $t \in \mathcal{O}_i$, we get
    \begin{align}
        \int_{\mathbb{R}^d}|\mathcal{G}(\varrho)|(t, \cdot)\; {\rm d}\bs{x} \leq 1.
    \end{align}
Finally, $\int_{\mathbb{R}^d}|\mathcal{G}(\varrho)|(1, \cdot)\; {\rm d}\bs{x} = 1$.
\end{subequations}
Putting everything together leads to
\begin{align}
\norm{\mathcal{G}(\varrho)}_{L^s_t L^1_{\bs{x}}} \leq 2^{-1/s}(1 +  \norm{\varrho}_{L^s_t L^1_{\bs{x}}})\, .
\end{align}
Moreover, it is immediate to check that $\mathcal{G}: \mathcal{Y} \to \mathcal{Y}$ is a contraction:
\begin{equation}
    \norm{\mathcal{G}(\varrho) - \mathcal{G}(\tilde{\varrho})}_{L^s_tL^1_{\bs{x}}}
    \le 2^{-1/s}\norm{\varrho - \Tilde{\varrho}}_{L^s_tL^1_{\bs{x}}} \, ,
    \quad \text{for every $\varrho,\Tilde{\varrho}\in \mathcal{Y}$}\, .
\end{equation}
We estimate the $\text{Lip}_t W^{-1, q}_{\bs{x}}$ norm.
\begin{subequations}
    When $t \in \mathcal{E}_i$, we get
\begin{align}
\norm{\mathcal{G}(\varrho)-\mathcal{G}(\tilde\varrho)}_{\text{Lip}(\mathcal{E}_i, W^{-1, q}_{\bs{x}})} 
&
\leq \frac{1}{\tau_i^{mid} - \tau_i} \left(\frac{1}{2^{(1+\nu)i}}\right)^{1 - \frac{d}{q^\prime}} \norm{\varrho-\tilde\varrho}_{\text{Lip}_t W^{-1, q}_{\bs{x}}},
\end{align}
\begin{align}
\norm{\mathcal{G}(\varrho)}_{\text{Lip}(\mathcal{E}_i, W^{-1, q}_{\bs{x}})} 
&
\leq \frac{1}{\tau_i^{mid} - \tau_i} \left(\frac{1}{2^{(1+\nu)i}}\right)^{1 - \frac{d}{q^\prime}} \norm{\varrho}_{\text{Lip}_t W^{-1, q}_{\bs{x}}}.
\end{align}
        In the time interval $\mathcal{O}_i$, we use that $\mathcal{G}(\varrho)$ is explicit and solves a transport equation {since it is made of a sum of suitable rescalings of $\rho_b$, which in turn solves the transport equation thanks to Proposition \ref{prop: building block vector field}}. For every test function $\varphi \in C^\infty_c(\mathbb{R}^d)$, it holds by the continuity equation, H\"older inequality, and since ${\supp \ol{\rho}_{b, i}(t, \cdot)}$ has measure ${2^{-\nu d i}}$
\begin{align}
 \left|\frac{d}{dt} \int_{\mathbb{R}^d} \ol{\rho}_{b, i} \, \varphi \; {\rm d} \bs{x}\right| & =  \left|- \int_{\mathbb{R}^d} \ol{\rho}_{b, i} \, \ol{\bs{v}}_{b, i} \, \nabla \varphi \; {\rm d} \bs{x}\right| \nonumber \\
 & \leq \norm{\ol{\bs{v}}_{b, i}}_{L^\infty_{t, \bs{x}}} \norm{\ol{\rho}_{b, i}}_{L^\infty_{t, \bs{x}}} \int_{\supp \ol{\rho}_{b, i}(t, \cdot)} |\nabla \varphi| \; {\rm d} \bs{x} \nonumber \\
& \leq C(\nu, d)  \frac{1}{2^i} \frac{2^{\nu d i}}{\tau_{i+1} - \tau^{mid}_i} \left(\frac{1}{2^{\nu d i}} \right)^{\frac{1}{q}} \norm{\nabla \varphi}_{L^{q^\prime}}.
\end{align}
Hence,
\begin{align}
\norm{\mathcal{G}(\varrho)}_{\text{Lip}(\mathcal{O}_i, W^{-1, q}_{\bs{x}})} \leq C(\nu, d)  \frac{2^{i(\nu  \frac{d}{q'} - 1)}}{\tau_{i+1} - \tau^{mid}_i}  \, .
\end{align}
\label{lipW-1q}
\end{subequations}
$\mathcal{G}(\varrho)$ is continuous at the interfaces $\tau_i$, $\tau_i^{mid}$ for every $i\in \mathbb{N}$ since the right and left limits at these points are explicit, coming either as a rescaling of $\rho^{in}$ and $\rho^{out}$ or from suitable rescalings of \eqref{eqn:rhose}. 
Hence, when \eqref{eq3} holds, we have
\begin{align}\label{eq:z1}
\norm{\mathcal{G}(\varrho)}_{\text{Lip}_t, W^{-1, q}_{\bs{x}}} 
&\leq C(\nu, d) 
\sup_{i \in \mathbb{N}} \max\left\{2^{i(\nu \frac{d}{q'} + \beta-1)}, \; 2^{i\beta} 2^{i(1+\nu)(\frac{d}{q'}-1)} \right\} \max\left\{1, \norm{\rho}_{\text{Lip}_t, W^{-1, q}_{\bs{x}}}\right\}
\\& \leq C(\nu, d)
 \sup_{i \in \mathbb{N}} 2^{i(\nu \frac{d}{q'} + \beta-1)} \max\left\{1, \norm{\rho}_{\text{Lip}_t, W^{-1, q}_{\bs{x}}}\right\}< \infty \, .
\end{align}
To conclude that $\mathcal{G}(\rho)\in \mathcal{X}$ we need to check that $\mathcal{G}(\rho)(1,\bs{x})=\rho^{out}$. The limit $\lim_{t\to 1} \mathcal{G}(\rho)(t, \cdot)$ exists in $W^{-1,q}$ as a consequence of the Lipschitz bound \eqref{eq:z1}, on the other hand $(\mathcal{G}(\rho)(\tau_i,\cdot))_{i\in \N}$ is explicit and converges in the sense of distributions to $\rho^{out}$. Since the distributional limit must coincide with the $W^{-1,q}$ limit, we conclude that $\mathcal{G}(\rho)(1,\cdot)=\rho^{out}$.

{To ensure that $\mathcal{G}: \mathcal{X} \to \mathcal{X}$ is a contraction, by \eqref{lipW-1q}{\color{blue}(a)} we impose
\begin{equation}
   \sup_{i\in \mathbb{N}} \frac{1}{\tau_i^{mid} - \tau_i} \left(\frac{1}{2^{(1+\nu)i}}\right)^{1 - \frac{d}{q^\prime}} < 1 \, ,
\end{equation}
which is certainly satisfied if $\nu \ge 2$ and \eqref{eq3} holds.}
\end{proof}

\subsection{Proof of Theorem \ref{thm:nonuniq L1}}
We fix $1<p<d$, and $1<s<\infty$. There exist $0<\beta<1$, $\nu \ge \nu_0$, $\alpha\in (0,1)$ and $q>1$ such that both $\mathcal{F}$ and $\mathcal{G}$ admit unique fix points
\begin{equation}
    \bs{v}\in C([0,1]; \dot W^{1,p})\cap C^\alpha \, , \quad
    \rho\in \text{Lip}([0,1]; W^{-1,q})\cap L^s([0,1];L^1) \, .
\end{equation}


We show that $\rho, \bs{v}$ solve the continuity equation, namely
\begin{align}
\frac{d}{dt} \int_{\mathbb{R}^d} \rho \, \phi \, {\rm d} \bs{x} = - \int_{\mathbb{R}^d} \rho \, \bs{v} \cdot \nabla \phi \, {\rm d} \bs{x} \quad \mbox{for all }  \phi \in C^\infty_c(\mathbb{R}^d),
\label{density solves continuity: weak sense for L1}
\end{align}
for a.e. $t \in [0, 1]$. Notice that the left-hand side of \eqref{density solves continuity: weak sense for L1} exists for a.e. $t \in [0, 1]$ since the function $t \to \int_{\mathbb{R}^d} \rho \, \phi \, {\rm d} \bs{x} $ is Lipschitz by $ \rho\in \text{Lip}([0,1]; W^{-1,q})$.
The validity of \eqref{density solves continuity: weak sense for L1}  would be enough to conclude that $\rho, \bs{v}$ solve the continuity equation
.

To prove \eqref{density solves continuity: weak sense for L1}  for a.e. $t \in [0, 1]$, 
 we let $\Sigma\subset [0,1]$ be the set of those times such that \eqref{density solves continuity: weak sense for L1} holds. We observe that, if $\rho, \bs{v}$ satisfy \eqref{density solves continuity: weak sense for L1} on  $ \Sigma$, from the definition of $\mathcal{F}$ and $\mathcal{G}$ we have that $\mathcal G(\rho)$ and $\mathcal F(\bs{v})$ satisfy \eqref{density solves continuity: weak sense for L1} on $\cup_{i\in \N} \mathcal{O}_i$ as well as in a countable union of rescaled copies of $ \Sigma$, namely on a set of measure $\frac{1}{2}\mathscr{L}^1( \Sigma) + \frac{1}{2}$.
Since $\rho$ and $\bs{v}$ are fixed points, we deduce that $\mathscr{L}^1(\Sigma) \ge  \frac{1}{2}\mathscr{L}^1(\Sigma) + \frac{1}{2}$, hence $\mathscr{L}^1(\Sigma) = 1$. 

Since $\rho \geq 0$ is a distributional solution of the continuity equation compactly supported in space, its space integral, which coincides with the $L^1$ norm, is constant in time, showing in particular improved integrability of $\rho\in L^\infty_t L^1_{\bs{x}}$.

\section{Nonuniqueness in $L^\infty_t L^r_{\bs{x}}$} 
\label{sec: outline}









The nonunique solution  $\rho \in L^\infty_t L^1_{\bs{x}}$ constructed in Section \ref{sec: p=1} does not belong to $L^\infty_t L^r_{\bs{x}}$ for any $r>1$. Indeed, by mass conservation and nonnegativity of $\rho$, 
$
    \| \rho(t,\cdot)\|_{L^1} = 1$ 
for all $t\in [0,1]$,
and the measure of the support of $\rho(t, \cdot)$ shrinks to zero as $t\to 1$. 
This observation is consistent with plots in figure \ref{fig: plot Lr norm density L1 consturction}.

The key idea to gain more integrability in the density field  is to avoid concentration all at the same time in space with a heterogeneous-in-space construction. We also need to change our initial density $\rho^{\rm in}$;
in the $L^1$ construction, a single cube divides into $2^d$ smaller cubes, in the $L^r$ construction, a single cube breaks into $2^{\eta d}$ smaller cubes, where $\eta \in \mathbb{N}$ is a new parameter to be fixed later. 

\subsection{Description of the asynchronous evolution}

Figure \ref{fig: density evolution Lp} shows a few snapshots of the density solution we have in mind. These snapshots are arranged in an increasing time order going from $t=0$ to $t=1$. For the time being, we do not worry about the precise times of these snapshots and rather focus on the essence of the density evolution. In the frame (a) (at $t=0$), the density field concentrates on $2^{\eta d}$ cubes each of size $\frac{1}{2^{\eta + \nu}}$, just as in the $L^1$ construction with $\eta = 1$. However, unlike the $L^1$ case, at a given time, only one of the cube out $2^{\eta d}$ cubes breaks into $2^{\eta d}$ smaller cubes each of size $\frac{1}{2^{2(\eta + \nu)}}$. We call this as \textit{asynchronous breaking} of cubes. The idea of asynchronous breaking of cubes does not just stop at the first generation but carries over to all generations. For instance, frame (d) shows asynchronous breaking at the second generation. 
 Once a cube of $i$th generation breaks into $2^{\eta d}$ little cubes, we evenly spread the little cubes over the $i$th generation dyadic cube using the building block vector field from section \ref{sec: building block}.

\bigskip

 Continue applying the procedure of asynchronous breaking followed by spreading of the cubes leads to the emergence of Lebesgue density of magnitude one in a localized portion the domain, which then steadily proliferates to other parts eventually covering the whole domain. This is in contrast with the $L^1$ case where the density becomes unity only at the final instance $t = 1$.




\subsection{The sharp range: Heuristic}

 Above in Section~\ref{sec: p=1}, we constructed $b\in L^\infty_t W^{1,p}_x$ and a density $ \rho \in L^\infty_t L^1_x$ with the following structure. At almost every time $t\in (0,1)$, there is an integer $i \geq 1$ such that for $\eta=1$ 
\begin{itemize}
\item
$\rho(t, \cdot )$ is concentrated on the disjoint union of $2^{\eta di}$ cubes of size $2^{-(\eta+\nu)i}$. On each cube, $\rho(t, \cdot )$ assumes the value $2^{d \nu i}$.

\item $ b(t, \cdot )$ the sum of $2^{\eta di}$ building block vector fields with disjoint supports of size comparable to  $2^{-(\eta+\nu)i}$, and magnitude roughly $2^{(\beta-\eta)i}$ (so that they can translate a cube of side $2^{-(\eta+\nu)i}$ at distance $2^{-\eta}$ in time $t=2^{-\beta}$).
\end{itemize}
In order to reach the sharp  integrability bound of Theorem \ref{Intro: main thm}, we perform an asyncronizazion as follows. 
At almost every time $t\in (0,1)$, the support of the density is the disjoint union of cubes of different sizes, but only $2^{\eta d}$ of them with a fixed size $2^{-(\eta+\nu)i}$ are moving. However, their speed increases to $2^{d\eta i}$, so that the total time needed to move all the cubes of size $2^{-(\eta+\nu)i}$ is unchanged. Therefore, our new construction will have the following two properties at almost every time $t$:

 \begin{figure}[H]
\centering
 \includegraphics[scale = 0.37]{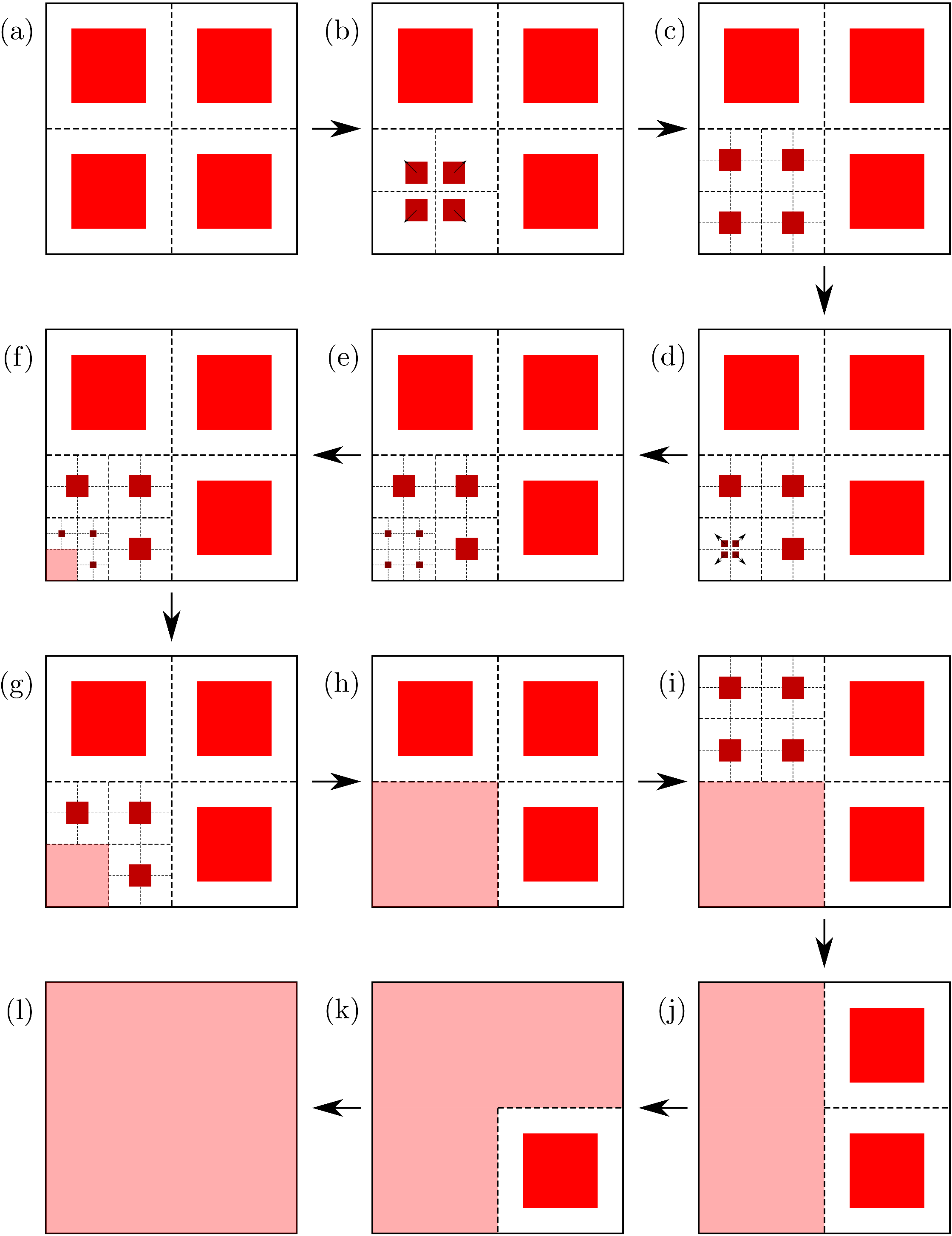}
 \caption{Panel (a) to (l) illustrates the evolution of the proposed density solution based on the heterogeneous construction described in the main text. The arrows between the panels denote the direction of the time. The figure focuses on the nature of the density evolution rather than specifying the precise time instance of each snapshot. Going from panel (a) to (b) and from (h) to (i) shows the asynchronous breaking of the first generation cubes and going from panel (c) to (d) shows the asynchronous breaking of the second generation cubes.}
 \label{fig: density evolution Lp}
\end{figure}

\begin{itemize}
\item
the support of $\rho(t, \cdot)$ is the union of cubes of different sizes. For every $i\in \N$, there are at most $2^{\eta d}$ cubes of size $2^{-(\eta+\nu)i}$, and on these cubes the constant density is $2^{d \nu i}$.

\item There exists $i \in \N$, such that $b(t, \cdot)$ is the sum of $2^{\eta d}$ building block vector fields with disjoint supports of size comparable to $2^{-(\eta+\nu)i}$, with magnitude $2^{(\beta-\eta + d\eta)i}$.
\end{itemize}

While the precise modification of the construction is technical, these two properties alone allow us to quickly compute the $L^r$ norm of the density and the $W^{1,p}$ norm of the velocity field for a.e. $t\in [0,1]$: we have
  \begin{align}
 \norm{\rho(t, \cdot)}_{L^r} \leq 2^{\eta d}  \sum_{i = 1}^{\infty} \frac{2^{i \nu d}}{2^{i(\eta + \nu)\frac{d}{r}}} \, ,
  \label{heuristic: Lr norm den}
 \end{align}
\begin{align}
\norm{\bs{v} (t, \cdot)}_{W^{1, p}} \lesssim  2^{i \beta} 2^{i (\eta d + \nu)} \frac{1}{2^{i(\eta + \nu) \frac{d}{p}}},
\end{align}
where $i \in \N$ depends on $t$.
Therefore, $\rho \in L^\infty_t L^r_{\bs{x}}$ and $\bs{v}  \in  L^\infty_t W^{1,p}_{\bs{x}}$ if
 \begin{align}
\frac 1 r > \frac{\nu}{\eta + \nu}, \qquad \frac 1 p \geq \frac{\eta d + \nu + \beta}{(\eta + \nu)d}.
\label{heuristic: r range}
 \end{align}
Combining the two bounds we see the appearence of the bound \eqref{cond:ip}
$$
 \frac 1 p + \frac {d-1} {dr} > \frac{(\eta  + \nu )d + \beta}{(\eta + \nu)d} \to 1 \qquad \mbox{as }\beta\to 0,
$$
and viceversa given $p$ and $r$ as in \eqref{cond:ip} it is easy to find $\beta$ small and $\nu, \eta$ such that \eqref{heuristic: r range} holds. 

The need for the parameter $\eta$ is justified by the following: we require $\nu \geq 1$ to close the fixed point argument, so that $\eta=1$ in the first inequality \eqref{heuristic: r range} would restrict $r < 2$, while the freedom in $\eta$ still allows to cover the entire range. 

\subsection{Time series}
\label{subsec: time series}
In our vector field construction, we work on the time interval $[0, 1]$. Given parameters $0 < \beta < 1$ and $\eta \in \mathbb{N}$, we define a few useful checkpoints in time. For $k \in \{1, \dots 2^{\eta d}\}$, we define 
\begin{align}
& \tau^k_1 \coloneqq \frac{(k-1)}{2^{\eta d}}, \nonumber \\
& \tau^k_2 \coloneqq \left(k - \frac{1}{2^\beta}\right) \frac{1}{2^{\eta d}}, \nonumber \\ 
& \tau^k_{mid} \coloneqq \frac{\tau^k_1 + \tau^k_2}{2}, \nonumber \\
& \tau^k_\infty = \frac{k}{2^{\eta d}}.
\end{align}
With the definition of checkpoints above, it is clear that for every $k \in \{1, \dots 2^{\eta d}\}$
\begin{align}
\tau^k_2 - \tau^k_1 = \frac{2^\beta - 1}{2^\beta} \frac{1}{2^{\eta d}} \quad 
\text{and} \quad
\tau^k_{\infty} - \tau^k_{2} = \frac{1}{2^\beta} \frac{1}{2^{\eta d}}
\end{align}
and for every $k \in \{1, \dots 2^{\eta d} - 1\}$
\begin{align}
\tau^{k+1}_1 = \tau^k_{\infty}.
\end{align}
With the help of these checkpoints, for $k \in \{1, \dots 2^{\eta d}\}$, we define a time interval, where the asynchronous action of the velocity field is concentrated only on the $k$th cube, as
\begin{align}
\mathcal{T}^k \coloneqq [\tau_1^k, \tau_{\infty}^k),
\end{align}
which is further divided into three time intervals as
\begin{align}
\mathcal{T}_{(1)}^k \coloneqq [\tau_1^k, \tau_{mid}^k], \qquad \mathcal{T}_{(2)}^k \coloneqq (\tau_{mid}^k, \tau_{2}^k), \quad
\text{and} \quad 
\mathcal{T}_{(3)}^k \coloneqq [\tau_2^k, \tau_{\infty}^k].
\end{align}

\begin{figure}[h]
\centering
 \includegraphics[scale = 0.4]{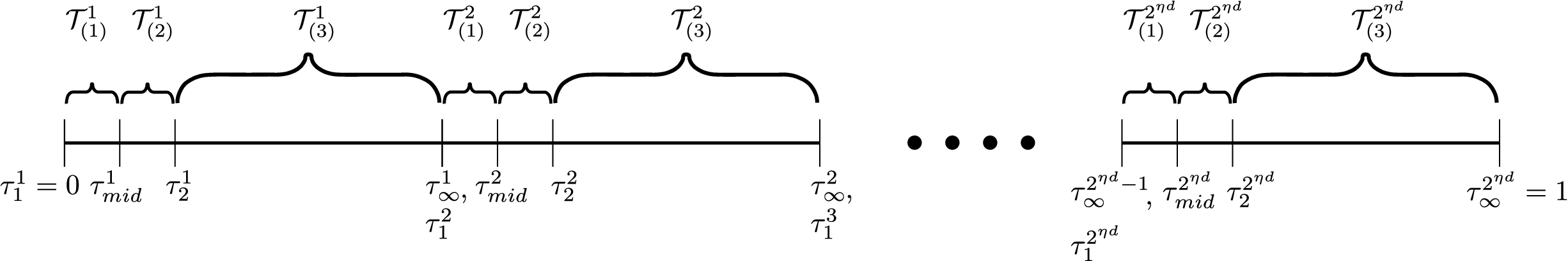}
 \caption{shows various checkpoints and time intervals defined in section \ref{subsec: time series}. The figure corresponds to $\beta = 2/3$, $\eta = 2$ and $d = 2$.}
 \label{fig: org prb}
\end{figure}


\subsection{The new fixed point argument}
\label{subsec: new fixed point}

The way we execute the fixed point argument to construct the vector field $\bs{v}$ and density $\rho$ is different in the $L^r$ case. In this section, we first explain the need to have a new fixed-point argument and then lay out the details of how this new argument works.   

 \begin{figure}[h]
\centering
 \includegraphics[scale = 0.28]{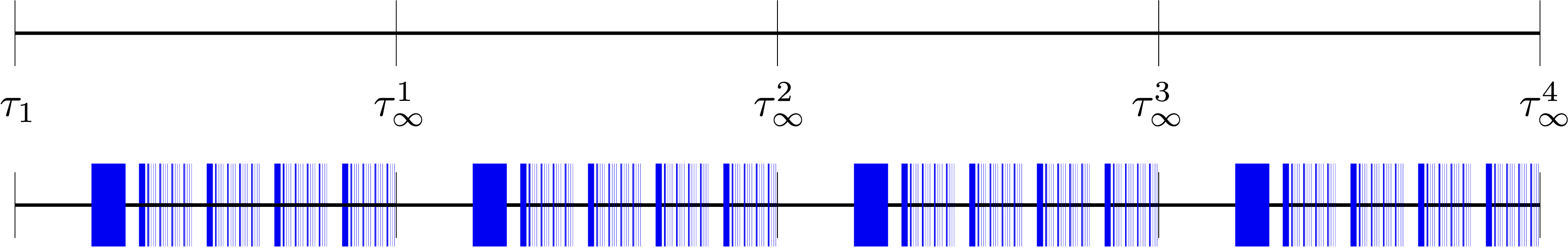}
 \caption{The blue lines show the time intervals where we would be required to explicitly define vector field (resp. the density field) in our fixed-point relation. 
}
 \label{fig: available definition Lp}
\end{figure}

We could not apply the fixed point argument the same way as in the $L^1$ case here. The difficulty is not a technical one. Rather, adapting the $L^1$ approach here makes the argument quite intricate. To illustrate the difficulty, imagine a tree of infinite depth whose root corresponds to the first-generation cubes, the second-level children correspond to the second-generation cubes, and so on. In the context of this tree, what we did in the $L^1$ case is akin to a breadth-first search; we translated all the first-generation cubes first, all the second-generation cubes after that, and so on. 

\begin{figure}[h]
\centering
 \includegraphics[scale = 0.4]{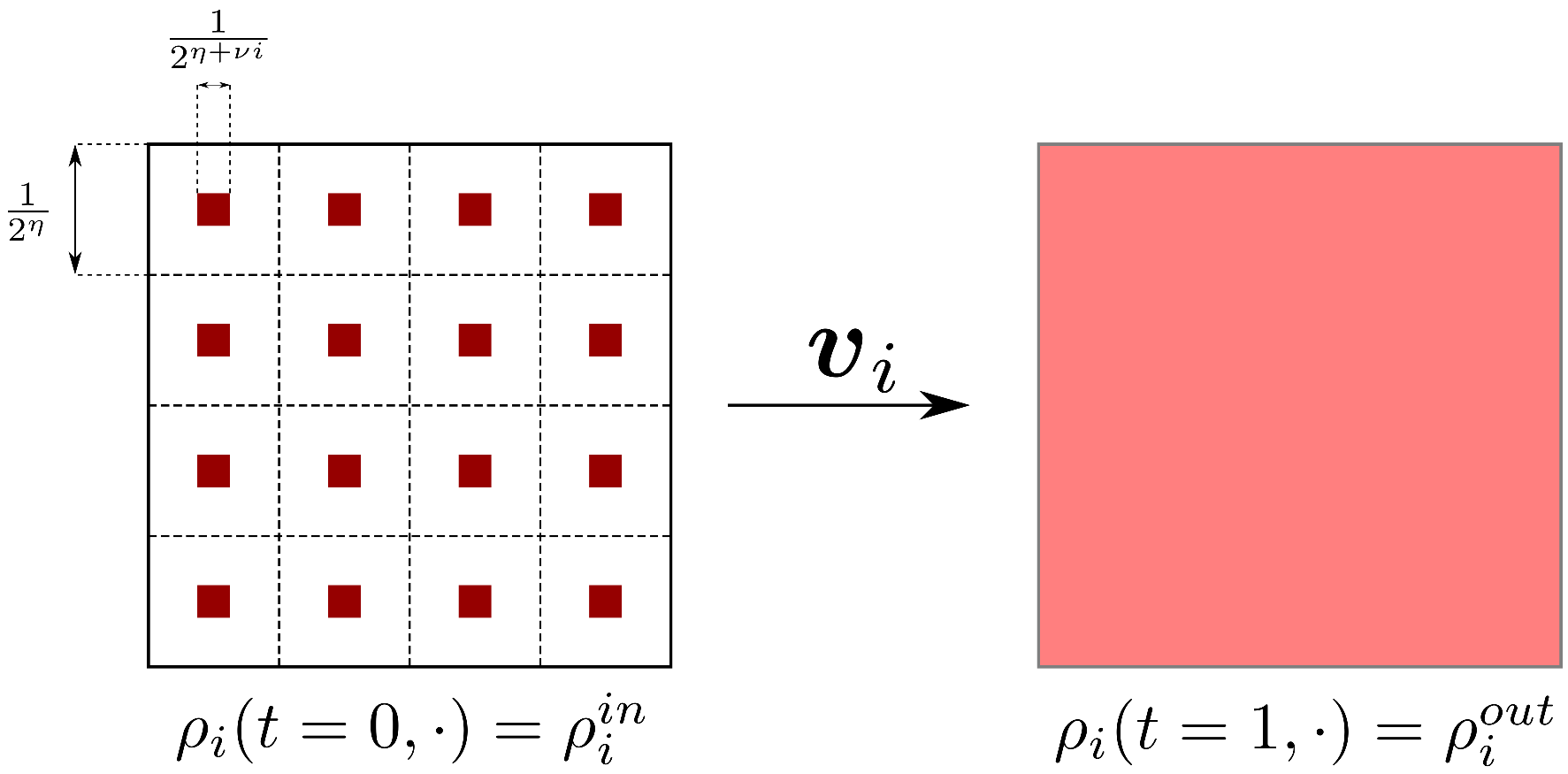}
 \caption{shows the initial density $\rho^{in}_i$ and final density $\rho^{out}_i$ that the solution of the transport equation $\rho_i$ (corresponding to the vector field $\bs{v}_i$) satisfies. At time $t = 0$, the density $\rho_i$ concentrates on $2^{\eta d}$ cubes of size $\frac{1}{2^{\eta + \nu i}}$ with magnitude $2^{\nu d i}$. At time $t = 1$, the density is the characteristic function of the unit cube. This figure corresponds to $i = 3$, $\nu = 2/3$, $\eta = 2$, $d = 2$.}
 \label{fig: org prb}
\end{figure}

This is also clear from figure \ref{fig: available definition of w}{\color{blue}(b)} which shows the time marker for the translation of cubes of different sizes. In this figure, the decreasing thickness of the blue line corresponds to the increasing generation of cubes. In this analogy, the construction in the $L^r$ case is akin to a depth-first search, where we move only one of the first-generation cubes first, and then only after translating all its descendent cubes, can we move the next first-generation cube. Similar time markers for the $L^r$ case are shown in figure \ref{fig: available definition Lp}. From this figure, the difficulty now becomes apparent. We will be required to define the vector field  explicitly using rescaled copies of the building block on a jumbled mess of time intervals (shown in blue in figure \ref{fig: available definition Lp}) corresponding to the motion of different generations of cubes. In addition, we will have to define the vector field implicitly in the remaining gaps (through rescaled copies depending on the gap size), which are mixed up as well.  This requires to find a new, suitable setup for the contraction. 

\bigskip

We overcome this difficulty by considering a fixed point argument in appropriate spaces of sequences of vector fields $\{\bs{v}_i\}_{i=1}^{\infty}$ and densities $\{\rho_i\}_{i=1}^{\infty}$. For a given $i \in \mathbb{N}$, $\bs{v}_i$ and $\rho_i$ solves the transport equation:
\begin{align}
\partial_t \rho_i + \bs{v}_i \cdot \nabla \rho_i = 0.
\end{align}
The different density solutions $\rho_i$ only differ in the initial conditions:



\begin{align}\label{eq:initial dens}
    \rho_i^{out}(x) \coloneqq 1_{\ol{Q}(\bs{0}, 1)}(x) \, ,
    \quad\quad 
    \rho^{in}_i \coloneqq \sum_{k= 1}^{2^{\eta d}} \rho^{c}_i(\bs{x} - \bs{c}^\eta_{k}) \, ,
    \quad \quad
    \rho_i^c \coloneqq 2^{\nu d i} 1_{\ol{Q}\left(\bs{0}, \frac{1}{2^{\eta + \nu i}}\right)} \, .
\end{align}

\begin{figure}[H]
\centering
 \includegraphics[scale = 0.25]{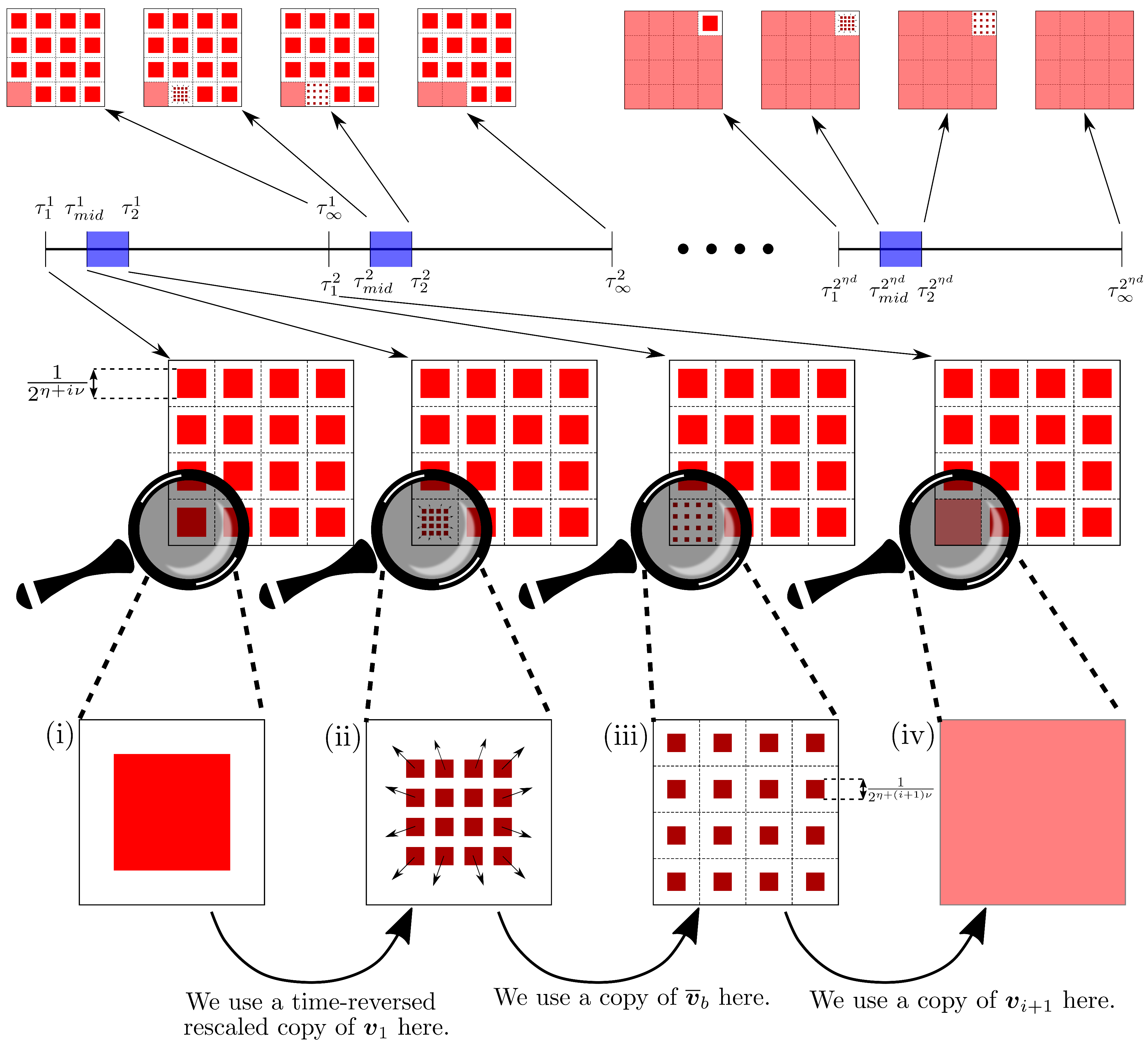}
 \caption{The figure shows the density $\rho_i$ at various checkpoints defined in section \ref{subsec: time series}. At time $\tau^1_1$, the density concentrates on $2^{\eta d}$ cubes each of size $\frac{1}{2^{\eta + i \nu}}$. All these cubes progressively (but independently in time) acquire Lebesgue density such that at $\rho_i(t = 1)=1_{\ol{Q}(0, 1)}$, the density becomes the characteristic function of the unit cube centered at zero. The second row shows how the left bottom most cube acquires the Lebesgue density on the time interval  $\mathcal{T}^1$. The panels (i), (ii), (iii) and (iv) shows zoomed in version (by a scale factor $2^{\eta}$) of the left bottom cube at time $\tau^1_1$, $\tau^1_{mid}$, $\tau^1_2$, $\tau^1_\infty$ respectively. We see that the red cubes in (ii) are $\frac{1}{2^{\eta + \nu}}$ factor smaller than the red cube in (i). Therefore, we use the time-reversed rescaled copy of $\bs{v}_1$ in going from (i) to (ii). More generally, we do this on time intervals $\mathcal{T}^k_{(1)}$. From (ii) to (iii), we use rescaled copies of the building block vector field from section \ref{sec: building block}. The cubes in (iii) are a $\frac{1}{2^{\eta + (i+1) \nu}}$ factor smaller than the external white cube in (iii) which transforms into Lebesgue density in (iv). Naturally, we use rescaled copy of $\bs{v}_{i+1}$ form (iii) to (iv).}
 \label{fig: v1 depend on v2}
\end{figure}

Now we give a brief overview of the fixed-point relation $\mathcal{F}$ whose input and output is a sequence of vector fields. A similar description also works for the density field. We define $\mathcal{F} = \{\mathcal{F}_i\}_{i=1}^{\infty}$ in such a way that on the time intervals $\mathcal{T}^k_{(1)}$, $\mathcal{F}_i(\{\bs{v}_i\}_{i=1}^{\infty})$ is implicitly defined as a rescaled copy of $\bs{v}_1$. On the intervals $\mathcal{T}^k_{(2)}$, $\mathcal{F}_i(\{\bs{v}_i\}_{i=1}^{\infty})$ is explicitly defined as a rescaled copy of the building block vector field $\ol{\bs{v}}_b$. Finally, on the time intervals $\mathcal{T}^k_{(3)}$, $\mathcal{F}_i(\{\bs{v}_i\}_{i=1}^{\infty})$ is implicitly defined as a rescaled copy of $\bs{v}_{i+1}$. It is only in the third time interval where the construction of $\bs{v}_i$ depends on the next vector field in the sequence, namely, $\bs{v}_{i+1}$. Figure \ref{fig: v1 depend on v2} illustrates this point, namely, how the construction of $\bs{v}_{i}$ depends on $\bs{v}_{i+1}$. 

\bigskip

Once the fixed point relation is defined, we show that the desired contraction happens in $\ell^\infty_{\gamma_1} C_t W^{1, p}_{\bs{x}}$ space (see section \ref{subsec:faf} for the exact definition). Here, $\ell^\infty_{\gamma_1}$ is weighted version of the $\ell^\infty$ space (see \ref{def: weighted l infinity}).

\bigskip


Recall, from section \ref{subsec: L1 distinguished points}, $\bs{c}_k^\eta$ are the centers of the cubes in the subdivision of $[-\frac{1}{2}, \frac{1}{2}]^d$ into cubes of side $2^{-\eta}$. As $\eta$ will be a fixed quantity in the following sections, we make a slight notational change and drop $\eta$ from the superscript, using $\bs{c}_k$ to mean $\bs{c}_k^\eta$ in rest of the sections.

\section{Construction of the vector field $\bs{v}$ as a Banach fixed point}
\label{sec: vector field }


\subsection{The functional analytic framework}\label{subsec:faf}

Let $p\ge 1$ and $\gamma_1 \in \mathbb{R}$ to be chosen later.
We denote by $\ell^{\infty}_{\gamma_1}(\mathbb{N})$ the weighted $\ell^\infty$ space 
endowed with the norm
\begin{align}
\norm{\{a_i\}_{i \in \mathbb{N}}}_{\ell^{\infty}_{\gamma_1}} \coloneqq \sup_{i \in \mathbb{N}} 2^{-\gamma_1 i}|a_i|\, .
\label{def: weighted l infinity}
\end{align}
Next, we define a set $\mathcal{X}$ of infinite sequences of vector fields as
\begin{align}
\mathcal{X} \coloneqq \Bigg\{ \{\bs{w}_i\}_{i \in \mathbb{N}} \in \ell^\infty_{\gamma_1}\left(\mathbb{N}; C([0, 1]; \dot W^{1, p}(\mathbb{R}^d, \mathbb{R}^d))\right)  \; \Bigg| \qquad \qquad \qquad \qquad \qquad \qquad \qquad \qquad \nonumber \\
\nabla \cdot \bs{w}_i \equiv 0, \; \bs{w}_i(0, \cdot) = \bs{w}_i(1, \cdot) \equiv \bs{0},  \; \supp_{\bs{x}} \bs{w}_i(t, \cdot) \subseteq \left[-\frac{1}{2}, \frac{1}{2}\right]^d \Bigg\},
\end{align}
equipped with the metric 
\begin{equation}
    \norm{\{\bs{w}_i\}_{i \in \mathbb{N}}}_{\ell^\infty_{\gamma_1} C_t \dot W^{1, p}_{\bs{x}}}
    = \sup_{i\in \mathbb N} 2^{-\gamma_1 i} \| \bs{w}_i \|_{C_t \dot W^{1,p}_x}\, ,
\end{equation}
where the homogeneous Sobolev norm is introduced in \eqref{eq:norms}. As observed in Remark \ref{rmk:L1 int}, each component of $\{\bs{w}_i\}_{i\in \N} \in \mathcal{X}$ automatically belongs to $L^p(\R^n)$.

\subsection{The iteration map}

Having in mind the description of the iteration step in subsection \ref{subsec: new fixed point}, we define the map $\mathcal{F}: \mathcal{X} \to \mathcal{X}$ whose fixed point will be the sought-after sequence of velocity fields.

\bigskip

We fix $\{\bs{w}_i\}_{i \in \mathbb{N}} \in \mathcal{X}$, $k \in \{1 , \dots , 2^{\eta d}\}$, and $i\in \mathbb{N}$. The $i$-th component of $\mathcal{F}(\{\bs{w}_i\}_{i \in \mathbb{N}})$ in the time interval $\mathcal{T}^k = [\tau_1^k, \tau_1^{k+1}]$ is defined separately on the three sub-intervals $\mathcal{T}^k_{(j)}$, $j=1,2,3$:

\begin{subequations}

\begin{itemize}

\item[(i)]  When $t \in \mathcal{T}^k_{(1)}$, 
\begin{align}
\mathcal{F}_i(\{\bs{w}_i\}_{i \in \mathbb{N}})(t, \bs{x}) 
\coloneqq 
\frac{1}{2^{\eta + \nu i}} \frac{1}{\tau^k_{mid} - \tau_1^k} \bs{w}_1 \left(\frac{\tau_{mid}^k - t}{\tau_{mid}^k - \tau^k_1},  2^{\eta + \nu i}(\bs{x} - \bs{c}_k)\right) \, .
\end{align}

\item[(ii)] When $t \in \mathcal{T}^k_{(2)}$, 
\begin{align}
\mathcal{F}_i(\{\bs{w}_i\}_{i \in \mathbb{N}})(t, \bs{x}) 
\coloneqq 
\frac{1}{2^{\eta}} \frac{1}{\tau^k_{2} - \tau_{mid}^k} \ol{\bs{v}}_b \left(\frac{t - \tau_{mid}^k}{\tau_{2}^k - \tau^k_{mid}},  2^{\eta}(\bs{x} - \bs{c}_k); \eta, \frac{1}{2^{\nu i}}, \frac{1}{2^{ \nu (i+1)}}\right) \, .
\end{align}

\item[(iii)] When $t \in \mathcal{T}^k_{(3)}$, 
\begin{align}
\mathcal{F}_i(\{\bs{w}_i\}_{i \in \mathbb{N}})(t, \bs{x}) 
\coloneqq 
 \frac{1}{2^{\eta}} \frac{1}{\tau^k_{\infty} - \tau_2^k} \bs{w}_{i+1} \left(\frac{t - \tau_{2}^k}{\tau_{\infty}^k - \tau^k_2},  2^{\eta}(\bs{x} - \bs{c}_k)\right)\, .
\end{align}
\end{itemize}
\end{subequations}
\label{eqn: vector field: def Fi}
We extend $\mathcal{F}_i(\{\bs{w}_i\}_{i \in \mathbb{N}})$ continuously up to time $t=1$ by setting $\mathcal{F}_i(\{\bs{w}_i\}_{i \in \mathbb{N}})(1, \bs{x}) \coloneqq \bs{0}$.

\bigskip

Let us briefly compare the definition of the map $\mathcal{F}$ and the heuristic explanation of the iteration process provided in subsection \ref{subsec: new fixed point}.

According to \eqref{eqn: vector field: def Fi}, in the time interval $\mathcal{T}^k$, the family of velocity fields $\mathcal{F}(\{\bs{w}_i\}_{i \in \mathbb{N}})$
acts only on the $k$th red cube. This is in perfect agreement with the idea of asynchronous move of cubes.

The action of $\mathcal{F}(\{\bs{w}_i\}_{i \in \mathbb{N}})$ results from three different sub-actions. In the first sub-interval $\mathcal{T}_{(1)}^k$, the vector field $\mathcal{F}_i(\{\bs{w}\}_{i \in \mathbb{N}})$ is defined using a reversed, rescaled copy of $w_1$. In the second interval $\mathcal{T}_{(2)}^k$, we utilize the forward blob field, appropriately rescaled. In the final interval $\mathcal{T}_{(3)}^k$, we employ the velocity field $\bs{w}_{i+1}$ forward in time. Specifically, if $\bs{w}_i$ is the vector field capable of moving the configuration with $2^{\eta d}$ squares, each with a side length of $2^{-\eta- i \nu}$, to match the Lebesgue measure, then $\mathcal F$ is going to fix $\{\bs{w}_i\}_{i\in \mathbb N}$, exactly as described in  subsection \ref{subsec: new fixed point}.

In particular, if we can prove the existence of a unique fixed point $\{w_i\}_{i\in \mathbb N}$ for the map $\mathcal{F}$, then $w_1$ would be the vector field we are looking for.

\begin{proposition}
Let $\mathcal{F}$ defined as in (\ref{eqn: vector field: def Fi}). Then given $d \geq 2$, $\eta \in \mathbb{N}$, $0 < \beta < 1$ and $\nu > \nu_0 \coloneqq 1 + \beta - \log_2(2^\beta - 1)$ there exists a $\gamma_1 \in \mathbb{R}$ such that $\mathcal{F}: \mathcal{X} \to \mathcal{X}$ is a contraction for
\begin{align}\label{eq: p range}
p < \frac{(\eta + \nu) d}{\eta d + \nu + \beta} \, .
\end{align}
\label{vector field: main prop}
\end{proposition}

\begin{proof}

We verify the boundary conditions and support of $\mathcal{F}_i(\{\bs{w}_{i^\prime}\}_{i^\prime \in \mathbb{N}})$ for any  $\{\bs{w}_{i^\prime}\}_{i^\prime \in \mathbb{N}} \in \mathcal{X}$. It is clear from the definition of $\mathcal{F}$ that for any $i \in \mathbb{N}$, we have
\begin{align}
\nabla \cdot \mathcal{F}_i(\{\bs{w}_{i^\prime}\}_{i^\prime \in \mathbb{N}}) \equiv 0, \quad \text{and} \quad \mathcal{F}_i(\{\bs{w}_{i^\prime}\}_{i^\prime \in \mathbb{N}})(0, \cdot) = \mathcal{F}_i(\{\bs{w}_{i^\prime}\}_{i^\prime \in \mathbb{N}})(1, \cdot) \equiv \bs{0} \, ,\quad \nonumber 
\end{align}
and that 
\begin{align}
\supp_{\bs{x}} \mathcal{F}_i(\{\bs{w}_{i^\prime}\}_{i^\prime \in \mathbb{N}})(t, \cdot) \subseteq \left[-\frac{1}{2}, \frac{1}{2}\right]^d \, .
\end{align}

We now estimate the $\dot W^{1,p}$-norm in space of $\mathcal{F}_i(\{\bs{w}_{i^\prime}\}_{i^\prime \in \mathbb{N}})$.
For $k \in \{1, \dots ,2^{\eta d}\}$ and $t \in \mathcal{T}^k_{(1)}$, a straightforward calculation shows that
\begin{align}
 \norm{\mathcal{F}_i(\{\bs{w}_{i^\prime}\}_{i^\prime \in \mathbb{N}})(t, \cdot) - \mathcal{F}_i(\{\wt{\bs{w}}_{i^\prime}\}_{i^\prime \in \mathbb{N}})(t, \cdot)}_{\dot W^{1, p}} 
&\leq 
\frac{1}{\tau^k_{mid} - \tau^k_1} \left(\frac{1}{2^{\eta + \nu i}}\right)^{\frac{d}{p}} \norm{\bs{w}_1-\wt{\bs{w}}_1}_{C_t \dot W^{1, p}_{\bs{x}}}, \nonumber \\
&\hspace{-10em}
\leq \frac{2^{\gamma_1}}{\tau^k_{mid} - \tau^k_1} \left(\frac{1}{2^{\eta + \nu i}}\right)^{\frac{d}{p}} \norm{\{\bs{w}_{i^\prime}\}_{i^\prime \in \mathbb{N}} - \{\wt{\bs{w}}_{i^\prime}\}_{i^\prime \in \mathbb{N}}}_{\ell^\infty_{\gamma_1} C_t \dot W^{1, p}_{\bs{x}}}.
\label{contraction: Fi est Tau 1}
\end{align}
When $t \in \mathcal{T}^k_{(2)}$, $\mathcal{F}_i(\{\bs{w}_{i^\prime}\}_{i^\prime \in \mathbb{N}})(t, \cdot) - \mathcal{F}_i(\{\wt{\bs{w}}_{i^\prime}\}_{i^\prime \in \mathbb{N}})(t, \cdot)\equiv 0$ and, by the properties of $\ol{\bs{v}}_b$ from Proposition \ref{prop: building block vector field}, we get
\begin{align}
\norm{\mathcal{F}_i(\{\bs{w}_{i^\prime}\}_{i^\prime \in \mathbb{N}})(t, \cdot)}_{\dot W^{1, p}} \leq \frac{1}{\tau^k_{2} - \tau^k_{mid}} \left( \frac{1}{2^{\eta}}\right)^{\frac{d}{p}} \norm{\ol{\bs{v}}_b}_{C_t \dot W^{1, p}_{\bs{x}}}, \nonumber \\
\leq C  \frac{1}{\tau^k_{2} - \tau^k_{mid}} \left(
\frac{1}{2^{2\eta + \nu(i+1)}}\right)^{\frac{d}{p} - 1}.
\label{contraction: Fi est Tau 2}
\end{align}

When $t \in \mathcal{T}^k_{(3)}$, then we get
\begin{align}
&\norm{\mathcal{F}_i(\{\bs{w}_{i^\prime}\}_{i^\prime \in \mathbb{N}})(t, \cdot)- \mathcal{F}_i(\{\wt{\bs{w}}_{i^\prime}\}_{i^\prime \in \mathbb{N}})(t, \cdot)}_{\dot W^{1, p}} =
\norm{\mathcal F_i(\{\bs{w}_{i^\prime}-\wt{\bs{w}}_{i^\prime}\}_{i^\prime \in \mathbb{N}})(t, \cdot)
}_{\dot W^{1, p}} 
\\& \hspace{4em}\leq 
\frac{1}{\tau^k_{\infty} - \tau^k_2} 
\left(\frac{1}{2^{\eta}}\right)^{\frac{d}{p}} \norm{\bs{w}_{i+1}- \wt{\bs{w}}_{i+1}}_{C_t \dot W^{1, p}_{\bs{x}}}, \nonumber 
\\& \hspace{4em} \leq 
\frac{2^{\gamma_1(i+1)}}{\tau^k_{\infty} - \tau^k_2} 
\left(\frac{1}{2^{\eta}}\right)^{\frac{d}{p}} \norm{\{\bs{w}_{i^\prime}\}_{i^\prime \in \mathbb{N}} - \{\wt{\bs{w}}_{i^\prime}\}_{i^\prime \in \mathbb{N}}}_{\ell^\infty_{\gamma_1} C_t \dot W^{1, p}_{\bs{x}}} .
\label{contraction: Fi est Tau 3}
\end{align}

{\bf Step 1: $\mathcal{F}$ maps $\mathcal{X}$ to $\mathcal{X}$.}
We show that $\mathcal{F}: \mathcal{X} \to \mathcal{X}$, provided
\begin{align}
\gamma_1 \geq \nu\left(1 - \frac{d}{p}\right)\, .
\label{contraction: F cond 1}
\end{align}

As the Sobolev norm of $\bs{w}_{i}(t, \cdot)$ and $\ol{\bs{v}}_b(t, \cdot)$ is continuous in time, the continuity of the Sobolev norm $\norm{\mathcal{\mathcal{F}}_i(\{\bs{w}_{i^\prime}\}_{i^\prime \in \mathbb{N}})(t, \cdot)}_{\dot W^{1, p}}$ in time is clear in the interiors of $\mathcal{T}^k_{(1)}$, $\mathcal{T}^k_{(2)}$ and $\mathcal{T}^k_{(3)}$. Now combining the definition of $\mathcal F_i$ from (\ref{eqn: vector field: def Fi}) with the information that $\bs{w}_i(0, \cdot) = \bs{w}_i(1, \cdot) \equiv \bs{0}$ and $\supp_t \ol{\bs{v}}_b \subseteq \left[\frac{1}{3}, \frac{2}{3}\right]$, we obtain the continuity of the Sobolev norm at the interfaces $\tau^k_1$, $\tau^k_{mid}$, $\tau^k_{2}$ and $\tau^k_{\infty}$.

Finally, noting estimates (\ref{contraction: Fi est Tau 1}), (\ref{contraction: Fi est Tau 2}), 
and (\ref{contraction: Fi est Tau 3}) and the continuity of the Sobolev norm in time gives
\begin{align}
& \frac{1}{2^{\gamma_1 i}} \norm{\mathcal{F}_i(\{\bs{w}_{i^\prime}\}_{i^\prime \in \mathbb{N}})}_{C_t \dot W^{1, p}_{\bs{x}}} \nonumber \\ 
& \qquad  \le C(\beta, \eta, \nu)  \sup_{i \in \mathbb{N}} \; \max\left\{\frac{1}{2^{\gamma_1 i}} \left(\frac{1}{2^{\nu i}}\right)^{\frac{d}{p}},  \frac{1}{2^{\gamma_1 i}} \left(\frac{1}{2^{\nu i}}\right)^{\frac{d}{p} - 1}\right\} \max\left\{1, \norm{\{\bs{w}_{i^\prime}\}_{i^\prime \in \mathbb{N}}}_{\ell^\infty_{\gamma_1} C_t \dot W^{1, p}_{\bs{x}}}\right\}.
\end{align}
Therefore, $\mathcal{F}(\{\bs{w}_{i^\prime}\}_{i^\prime \in \mathbb{N}}) \in \mathcal{X}$ if
\eqref{contraction: F cond 1} is satisfied.

\bigskip

{\bf Step 2: $\mathcal{F}$ is a contraction.}
We prove that, under the assumptions of Proposition \ref{vector field: main prop}, the map $\mathcal{F} : \mathcal{X} \to \mathcal{X}$ is a contraction.


Let $\{\bs{w}_{i^\prime}\}_{i^\prime \in \mathbb{N}}, \{\tilde{ \bs{w}}_{i^\prime}\}_{i^\prime \in \mathbb{N}} \in \mathcal{X}$. 
Combining \eqref{contraction: Fi est Tau 1} and \eqref{contraction: Fi est Tau 3} yields
\begin{align}
& \frac{1}{2^{\gamma_1 i}} \norm{\mathcal F_i(\{\bs{w}_{i^\prime}\}_{i^\prime \in \mathbb{N}}) - \mathcal F_i(\{\wt{\bs{w}}_{i^\prime}\}_{i^\prime \in \mathbb{N}})}_{C_t \dot W_{\bs{x}}^{1, p}} \nonumber \\
&  \qquad \leq \sup_{i \in \mathbb{N}} \max \left\{\frac{2^{\gamma_1}}{\tau^k_{mid} - \tau^k_1} \frac{1}{2^{\gamma_1 i}} \left(\frac{1}{2^{\eta + \nu i}}\right)^{\frac{d}{p}}, \frac{2^{\gamma_1}}{\tau^k_{\infty} - \tau^k_2} 
\left(\frac{1}{2^{\eta}}\right)^{\frac{d}{p}}\right\} 
\norm{\{\bs{w}_{i^\prime}\}_{i^\prime \in \mathbb{N}} - \{\wt{\bs{w}}_{i^\prime}\}_{i^\prime \in \mathbb{N}}}_{\ell^\infty_{\gamma_1} C_t \dot W^{1, p}_{\bs{x}}}
\nonumber \\
&\qquad  \leq 
2^{\gamma_1} \cdot 2^{\eta d} \cdot 2^{-\eta \frac{d}{p}} \cdot 2^\beta \cdot \sup_i \max \left\lbrace \frac{2}{2^\beta-1}\cdot 2^{-i(\gamma_1 + \nu \frac{d}{p})}, 1\right\rbrace
\norm{\{\bs{w}_{i^\prime}\}_{i^\prime \in \mathbb{N}} - \{\wt{\bs{w}}_{i^\prime}\}_{i^\prime \in \mathbb{N}}}_{\ell^\infty_{\gamma_1} C_t \dot W^{1, p}_{\bs{x}}}
\label{contraction: F contraction expression}
\end{align}
If $\gamma_1 \in \mathbb{R}$ obeys the condition (\ref{contraction: F cond 1}), the maximum above is realized at $i=1$. 
So, if we choose $\gamma_1 = \nu \big( 1-\frac{d}{p}\big)$, we get
\begin{equation}
    \sup_i \max \left\lbrace \frac{2}{2^\beta-1}\cdot 2^{-i(\gamma_1 + \nu \frac{d}{p})}, 1\right\rbrace
    = 
    \max \left\lbrace \frac{2^{1-\nu}}{2^\beta-1}, 1\right\rbrace
    = 1\, ,
\end{equation}
where in the last step we used the assumption $\nu> 1 + \beta - \log_2(2^\beta - 1)$. So, to get a contraction we have to impose
\begin{equation}
   2^{\nu\left(1 - \frac{d}{p}\right) + \eta d - \eta \frac{d}{p} + \beta} 
   =
   2^{\gamma_1} \cdot 2^{\eta d} \cdot 2^{-\eta \frac{d}{p}} \cdot 2^\beta 
   < 1
\end{equation}
which amounts to \eqref{eq: p range}.
\end{proof}



\section{Construction of the density field $\rho$}

We construct our density field in a manner similar to the previous section \S \ref{sec: vector field }. We first define the appropriate space of functions and a self-map. We then show that the mapping is a contraction.

\subsection{Functional analytic setting}

Let us fix parameters $r,q\ge 1$, and $\gamma_2, \gamma_3 \in \mathbb N$. Let $\rho_i^{\rm in}$ and $\rho^{\rm out}$ as in \eqref{eq:initial dens}. 
We define two Banach spaces $\mathcal{Y}$ and $\mathcal{Z}$ as
\begin{align}
\mathcal{Y} \coloneqq \ell^\infty_{\gamma_2}\left(\mathbb{N}; L^\infty([0, 1]; L^r(\mathbb{R}^d))\right)
\end{align}
equipped with the norm
\begin{equation}
\norm{\{ \varrho_i\}_{i\in \mathbb N}}_{\mathcal{Y}}
:=
\norm{\{ \varrho_i\}_{i\in \mathbb N}}_{\ell^\infty_{\gamma_2} L^\infty_t L^r_{\bs{x}}}
=
\sup_{i\in \mathbb N} 2^{-\gamma_2 i} \| \varrho_i \|_{L_t^\infty L_x^r} \, ,
\end{equation}
and
\begin{align}
\mathcal{Z} \coloneqq \Bigg\{ \{\varrho_i\}_{i \in \mathbb{N}} \in   \ell^\infty_{\gamma_3}\left(\mathbb{N}; \text{Lip}([0, 1]; W^{-1, q}(\mathbb{R}^d))\right) \; \Bigg| \qquad \qquad \qquad \qquad \qquad \nonumber \\
\varrho_i(0, \cdot) = \rho^{in}_i, \; \varrho_i(1, \cdot) = \rho^{out}, \; \supp_{\bs{x}} \varrho_i(t, \cdot) \subseteq \left[-\frac{1}{2}, \frac{1}{2}\right]^d \Bigg\}
\end{align}
equipped with the metric induced by
\begin{equation}
\norm{\{ \varrho_i\}_{i\in \mathbb N}}_{\mathcal{Z}}
:=
\norm{\{ \varrho_i\}_{i\in \mathbb N}}_{\ell^\infty_{\gamma_3} \text{Lip}_t W^{-1, q}_{\bs{x}}}
=
\sup_{i\in \mathbb N} 2^{-\gamma_3 i} \| \varrho_i \|_{{\rm Lip}_t W^{1,-q}_x}\, ,
\end{equation}
where we are using the notation introduced in \eqref{eq:norms2}, \eqref{eq:norms3}.

\subsection{Iteration map}

We now define a map $\mathcal{G}$ acting on sequences of densities. We will show that $\mathcal{G}$ is a contraction in both spaces $\mathcal{Y}$ and $\mathcal{Z}$ for appropriate choices of parameters.

\begin{subequations}
Fix $\{\varrho_i\}_{i \in \mathbb{N}} \in \mathcal{Y} \cup \mathcal{Z}$. For $k \in \{1, \dots ,2^{\eta d}\}$ and $i\in \mathbb N$, we define the $i$-th component of the self map as
\begin{align}
\mathcal{G}_i(\{\varrho_{i^\prime}\}_{i^\prime \in \mathbb{N}})(t, \bs{x}) \coloneqq \sum_{1 \leq k^\prime < k} \rho^{out}(2^\eta(\bs{x} - \bs{c}_{k^\prime})) + \sum_{k < k^\prime \leq 2^{\eta d}} \rho^{c}(2^\eta(\bs{x} - \bs{c}_{k^\prime}))
\end{align}
for every $\bs{x}\in \ol{Q}(\bs{0},\frac{1}{2})\setminus \ol{Q}(\bs{c}_k, \frac{1}{2^\eta})$ and $t\in \mathcal{T}^k$. In other words, when $t\in \mathcal{T}^k$, the $k'$-th cube of the subdivision with $k'< k$ is already in the final configuration. When $k'>k$, the $k'$-th cube is still in the initial configuration. The relevant part of the dynamics happens in the $k$-th cube: 

\begin{itemize}
    \item[(i)] If $t \in \mathcal{T}^k_{(1)}$, and $\bs{x}\in \ol{Q}(\bs{c}_k, \frac{1}{2^\eta})$ we set
\begin{align}
G_i(\{\varrho_{i^\prime}\}_{i^\prime \in \mathbb{N}})(t, \bs{x}) \coloneqq 
\; 2^{\nu d i} \varrho_1\left(\frac{\tau_{mid}^k - t}{\tau_{mid}^k - \tau^k_1},  2^{\eta + \nu i}(\bs{x} - \bs{c}_k)\right) \, .
\end{align}

    \item[(ii)] If $t \in \mathcal{T}^k_{(2)}$, and $\bs{x}\in \ol{Q}(\bs{c}_k, \frac{1}{2^\eta})$ we set
\begin{align}
\mathcal{G}_i(\{\varrho_{i^\prime}\}_{i^\prime \in \mathbb{N}})(t, \bs{x}) \coloneqq  2^{\nu d (i+1)} \ol{\rho}_b \left(\frac{t - \tau_{mid}^k}{\tau_{2}^k - \tau^k_{mid}},  2^{\eta}(\bs{x} - \bs{c}_k); \eta, \frac{1}{2^{\nu i}}, \frac{1}{2^{\nu (i+1)}}\right)\, .
\end{align}

    \item[(iii)] When $t \in \mathcal{T}^k_{(3)}$, and $\bs{x}\in \ol{Q}(\bs{c}_k, \frac{1}{2^\eta})$, we set
\begin{align}
\mathcal{G}_i(\{\varrho_{i^\prime}\}_{i^\prime \in \mathbb{N}})(t, \bs{x}) \coloneqq \varrho_{i+1}\left(\frac{ t - \tau_{2}^k}{\tau_{\infty}^k - \tau^k_2},  2^{\eta}(\bs{x} - \bs{c}_k)\right).
\end{align}
\end{itemize}
Finally, when $t = 1$, we set
\begin{align}
\mathcal{G}_i(\{\varrho_{i^\prime}\}_{i^\prime \in \mathbb{N}})(t, \bs{x}) \coloneqq \rho^{out}(\bs{x}) \, .
\end{align}
\label{density field: def Gi}
\end{subequations}

\begin{proposition}
\label{density field: main prop}
Let $\mathcal{G}$ be defined in (\ref{density field: def Gi}). Then for $d \geq 2$, $\eta \in \mathbb{N}$, $\beta \in (0, 1)$ and $\nu > 0$ 
\begin{enumerate}[label=(\roman*)]
    \item there exists a $\gamma_2 \in \mathbb{R}$ such that $\mathcal{G}:\mathcal{Y} \to \mathcal{Y}$ is a contraction for $r < \frac{\eta + \nu}{\nu}$,
    \item there exists a $\gamma_3 \in \mathbb{R}$ such that $\mathcal{G}:\mathcal{Z} \to \mathcal{Z}$ is a contraction for $1 < q < \frac{\eta d + \nu d}{\eta d + \nu d + \beta - \eta}$.
\end{enumerate}
\end{proposition}

\subsection{Contraction in $\mathcal{Y}$: proof of Proposition~\ref{density field: main prop}(i)}

We note that $\mathcal{G}(\cdot) - \mathcal{G}(0)$ is a linear operator. In particular, to check that $\mathcal{G}$ is a contraction on the linear space $\mathcal{Y}$, it is enough to check that
  $\| \mathcal{G}(0)\|_{\mathcal{Y}} < \infty$
and that 
    $\| \mathcal{G}(\{\varrho_i\}_{i\in \mathbb N}) - \mathcal{G}(0)\|_{\mathcal{Y}} 
    <
    \| \{\varrho_i\}_{i\in \mathbb N} \|_{\mathcal{Y}} $.

To prove that   $\| \mathcal{G}(0)\|_{\mathcal{Y}} < \infty$, we compute
\begin{equation}
    \| \mathcal{G}_i(0)(t,\cdot) \|_{L^r} \le 
    \frac{k-1}{2^{\frac{\eta d}{r}}} \| \rho^{\rm out}\|_{L^r}
    + 
    \frac{2^{\eta d} - k}{2^{\frac{\eta d}{r}}} \| \rho^{\rm c}\|_{L^r}
    +
    \frac{2^{\nu d (i+1)}}{2^{\frac{\eta d}{r}}} \norm{\ol{\rho}_b\left(\cdot, \frac{1}{2^{\nu i}}, \frac{1}{2^{\nu (i+1)}}\right)}_{L^\infty_t L^r_{\bs{x}}} \, .
\end{equation}
Noting that 
$
\norm{\rho^{out}}_{L^r} 
\leq 
\norm{\rho^{c}_i}_{L^r} 
= 
2^{\nu d i\frac{r-1}{r}} 
$ 
and
\begin{equation}
\norm{\ol{\rho}_b\left(\cdot, \frac{1}{2^{\nu i}}, \frac{1}{2^{\nu (i+1)}}\right)}_{L^\infty_t L^r_{\bs{x}}}
\le \left({2^{-\nu (i+1)}} \right)^{\frac{d}{r}} \, ,
\end{equation}
we get 
\begin{equation}
    \| \mathcal{G}(0) \|_{\mathcal{Y}}
    \le C(\eta, d,\nu) \sup_{i\in \mathbb N} 2^{\nu d \left(1 - \frac{1}{r}\right)i - \gamma_2 i}\, ,
\end{equation}
which is finite if and only if
\begin{equation}
    \gamma_2 \ge \nu d \left(1 - \frac{1}{r}\right) \, .
\end{equation}

We now focus on  the estimate  
    $\| \mathcal{G}(\{\varrho_i\}_{i\in \mathbb N}) - \mathcal{G}(0)\|_{\mathcal{Y}} 
    <
    \| \{\varrho_i\}_{i\in \mathbb N} \|_{\mathcal{Y}} $.Let $\{\varrho_{i^\prime}\}_{i^\prime \in \mathbb{N}} \in \mathcal{Y}$. We have
\begin{align}
    \| \mathcal{G}_i(\{\varrho_{i^\prime}\}_{i^\prime \in \mathbb{N}})(t, \cdot) - \mathcal{G}_i(0)(t,\cdot)\|_{L^r}
    &\le 
    \max\{
    2^{\nu d i} \left(\frac{1}{2^{\eta + \nu i}} \right)^{\frac{d}{r}} \norm{\varrho_1}_{L^\infty_t L^r_{\bs{x}}} ;
    2^{-\eta \frac{d}{r}} \norm{\varrho_{i + 1}}_{L^\infty_t L^r_{\bs{x}}} \}
    \\& \le
    2^{-\eta \frac{d}{r}} \max\{2^{\nu d(1-\frac{1}{r})i} 2^{\gamma_2} ; 2^{\gamma_2(i+1)}\} \| \{\varrho_{i'}\}_{i'\in \mathbb N}\|_{\mathcal{Y}} \, , 
\end{align}
hence,
\begin{align}
    \| \mathcal{G}_i(\{\varrho_{i^\prime}\}_{i^\prime \in \mathbb{N}})(t, \cdot) - \mathcal{G}_i(0)(t,\cdot)\|_{\mathcal{Y}}
    \le
    2^{-\eta \frac{d}{r} + \gamma_2} \cdot \sup_{i\in \mathbb N} \max\{ 2^{\nu d \left(1 - \frac{1}{r}\right)i - \gamma_2 i}; 1 \} \| \{\varrho_{i'}\}_{i'\in \mathbb N}\|_{\mathcal{Y}} \, .
\end{align}

To conclude, we pick $\gamma_2 = \nu d \left(1 - \frac{1}{r}\right)$, so that $\| \mathcal{\mathcal{G}}(0)\|_{\mathcal{Y}}<\infty$. Moreover,
\begin{align}
    \| \mathcal{\mathcal{G}}_i(\{\varrho_{i^\prime}\}_{i^\prime \in \mathbb{N}})(t, \cdot) - \mathcal{\mathcal{G}}_i(0)(t,\cdot)\|_{\mathcal{Y}}
    \le
    2^{-\eta \frac{d}{r} + \nu d(1-\frac{1}{r})} \| \{\varrho_{i'}\}_{i'\in \mathbb N}\|_{\mathcal{Y}} \, .
\end{align}
and we require the constant to be strictly less then $1$, namely
\begin{equation}
    -(\eta + \nu) \frac{d}{r} + \nu d < 0 \, .
\end{equation}

\subsection{Proof of Proposition~\ref{density field: main prop}(ii): $\mathcal{G}$ maps $\mathcal{Z}$ to $\mathcal{Z}$}


We begin by proving that $\mathcal{G}$ maps $\mathcal{Z}$ to $\mathcal{Z}$.


As regards the boundary conditions and support, from the definition of $\mathcal{G}_i$, it is easy to verify that for all  $\{\varrho_{i^\prime}\}_{i^\prime \in \mathbb{N}} \in \mathcal{Z}$
\begin{enumerate}[label={(\roman*)}]
\item $\mathcal{G}_i(\{\varrho_{i^\prime}\}_{i^\prime \in \mathbb{N}})(0, \cdot) = \rho^{in}_i$ after noting that $2^{\nu d i} \varrho_1(1, 2^{\eta + \nu i}(\bs{x} - \bs{c}_1)) = \rho^c_i(2^{\eta}(\bs{x} - \bs{c}_1))$,
\item $\mathcal{G}_i(\{\varrho_{i^\prime}\}_{i^\prime \in \mathbb{N}})(1, \cdot) = \rho^{out}_i$,
\item $\supp_{\bs{x}} \mathcal{G}_i(\{\varrho_{i^\prime}\}_{i^\prime \in \mathbb{N}})(t, \cdot) \subseteq \left[-\frac{1}{2}, \frac{1}{2}\right]^d$.
\end{enumerate}


 We first show that $\mathcal{G}_i(\{\varrho_{i^\prime}\}_{i^\prime \in \mathbb{N}})$ is Lipschitz in time in the $W^{-1, q}$ topology in the interior of time intervals $\mathcal{T}^k_{(1)}$, $\mathcal{T}^k_{(2)}$ and $\mathcal{T}^k_{(3)}$.
Then, we check that $\mathcal G_i(\{\varrho_{i^\prime}\}_{i^\prime \in \mathbb{N}})$ is continuous at the interfaces $\tau^k_1$, $\tau^k_{mid}$, $\tau^k_2$ and $\tau^k_\infty$.


\vspace{0.5cm}

\noindent
From the definition of $W^{-1,q}$ norm we deduce
\begin{align}
    \norm{\mathcal{G}_i(\{\varrho_{i^\prime}\}_{i^\prime \in \mathbb{N}})}_{C^1(\mathcal{T}^k_{(1)}; W^{-1,q})}
    \le &
    \frac{2^{\nu d i}}{\tau^k_{mid} - \tau^k_1} \left(\frac{1}{2^{\eta + \nu i}}\right)^{1 + \frac{d}{q}} \norm{\varrho_1}_{\text{Lip}_t W^{-1, q}_{\bs{x}}}
    \\
    \norm{\mathcal{G}_i(\{\varrho_{i^\prime}\}_{i^\prime \in \mathbb{N}})}_{C^1(\mathcal{T}^k_{(2)}; W^{-1,q})}
    \le &
    \frac{2^{\nu d (i+1)}}{\tau^k_2 - \tau^k_{mid}}  \left(\frac{1}{2^{\eta}} \right)^{1+\frac{d}{q}} \norm{\bar \rho_b \left(\cdot\, , \cdot\, ; \eta, \frac{1}{2^{\nu i}}, \frac{1}{2^{\nu(i+1)}}\right)}_{C^1_t W^{-1,q}_x}
    \\
    \le & C(\nu,d,\eta,\beta) 2^{\nu d i(1-\frac{1}{q})}
    \\
    \norm{\mathcal{G}_i(\{\varrho_{i^\prime}\}_{i^\prime \in \mathbb{N}})}_{C^1(\mathcal{T}^k_{(3)}; W^{-1,q})}
    \le &
    \frac{1}{\tau^k_{\infty} - \tau^k_2} \left(\frac{1}{2^{\eta}}\right)^{1 + \frac{d}{q}} \norm{\varrho_{i+1}}_{\text{Lip}_t W^{-1, q}_{\bs{x}}} \, .
\end{align}
In the second estimate we used that $\bar \rho_{b,i}(t, \bs{x}):= \bar \rho_{b}(t, \bs{x};\eta,\frac{1}{2^{\nu i}}, \frac{1}{2^{\nu(i+1)}})$ solves the transport equation with velocity field $ \ol{\bs{v}}_{b,i}(x,t):=  \ol{\bs{v}}_{b}(t, \bs{x};\eta,\frac{1}{2^{\nu i}}, \frac{1}{2^{\nu(i+1)}})$, hence  
\begin{align}
\left|\frac{d}{dt}\int_{\mathbb{R}^d} \ol{\rho}_{b, i} \, \varphi \; {\rm d} \bs{x}\right| & \leq \norm{\ol{\bs{v}}_{b, i}}_{L^\infty_{t, \bs{x}}} \norm{\ol{\rho}_{b, i}}_{L^\infty_{t, \bs{x}}} \int_{\supp \ol{\rho}_{b, i}(t, \cdot)} |\nabla \varphi| \; {\rm d} \bs{x} \nonumber \\
& \leq  
C \mathscr{L}^d(\supp \ol{\rho}_{b, i}(t, \cdot))^{\frac{1}{q}} \| \nabla \varphi \|_{L^{q'}}
\\& \leq 
C 2^{-d\nu i \frac{1}{q}} \| \nabla \varphi \|_{L^{q'}}\, ,
\end{align}
for any $\varphi \in C_c^\infty(\mathbb R^d)$.
We now check continuity of the $W^{-1, q}$ norm at $t = \tau^k_1$, $\tau^k_2$, $\tau^k_3$ and $\tau^k_\infty$. 
\vspace{0.5cm}

\noindent
After substituting $t = \tau^k_1$ in the expression of $\mathcal{G}_i$, we get
\begin{align}
\mathcal{G}_i(\{\varrho_{i^\prime}\}_{i^\prime \in \mathbb{N}})(\tau^k_1, \bs{x}) &= \sum_{1 \leq k^\prime < k} \rho^{out}(2^\eta(\bs{x} - \bs{c}_{k^\prime})) + \sum_{k < k^\prime \leq 2^{\eta d}} \rho^{c}_i(2^\eta(\bs{x} - \bs{c}_{k^\prime})) + 2^{\nu d i} \varrho_1\left(1,  2^{\eta + \nu i}(\bs{x} - \bs{c}_k)\right)
\nonumber
\\&
= \sum_{1 \leq k^\prime < k} \rho^{out}(2^\eta(\bs{x} - \bs{c}_{k^\prime})) + \sum_{k \leq k^\prime \leq 2^{\eta d}} \rho^{c}_i(2^\eta(\bs{x} - \bs{c}_{k^\prime})).
\end{align}
where in the last line we used that 
$2^{\nu di}\varrho_1\left(1,  2^{\eta + \nu i}(\bs{x} - \bs{c}_k)\right) = \rho^c_i(2^\eta(\bs{x} - \bs{c}_k))
$.

We now see that if $k = 1$ then $\mathcal{\mathcal{G}}_i(\{\varrho_{i^\prime}\}_{i^\prime \in \mathbb{N}})$ is equal to the initial condition $\rho^{in}_i$. For $k \geq 2$, then we also need to check for the continuity from the left at $t = \tau^k_1$. Since
$
\varrho_{i+1}(1, 2^\eta(\bs{x} - \bs{c}_k)) = \rho^{out}(1, 2^\eta(\bs{x} - \bs{c}_k))
$, 
we have that
\begin{align}
\mathcal{G}_i(\{\varrho_{i^\prime}\}_{i^\prime \in \mathbb{N}})(t, \bs{x}) \to  \sum_{1 \leq k^\prime < k} \rho^{out}(2^\eta(\bs{x} - \bs{c}_{k^\prime})) + \sum_{k \leq k^\prime \leq 2^{\eta d}} \rho^{c}_i(2^\eta(\bs{x} - \bs{c}_{k^\prime})) \quad \text{in} \quad W^{-1, q} \quad \text{as} \quad t \nearrow \tau^k_{1}.
\end{align}


The continuity at $t = \tau^k_{mid}$ can be verified as follows: 
\begin{align}
2^{\nu d i}\rho_{1}\left(0, 2^{\nu i}\bs{x}\right) & = 2^{\nu d i} \rho^{in}_1(2^{\nu i} \bs{x})  = \sum_{k^\prime = 1}^{2^{\eta d}} 2^{\nu d i} \rho^c_i(2^\eta(2^{\nu i} \bs{x} - \bs{c}_{k^\prime}))   = 
2^{\nu d (i + 1)} \ol{\rho}_b\left(0, \bs{x}; \frac{1}{2^{\nu i}}, \frac{1}{2^{\nu(i+1)}}\right).
\end{align}

The continuity at $t = \tau^k_{2}$ follows by a similar computation
\begin{align}
& \rho_{i+1}(0, \bs{x}) = \sum_{k^\prime = 1}^{2^{\eta d}}  \rho^c_{i+1}(2^\eta( \bs{x} - \bs{c}_{k^\prime})) 
 = 2^{\nu d (i + 1)} \ol{\rho}_b\left(1, \bs{x}; \frac{1}{2^{\nu i}}, \frac{1}{2^{\nu(i+1)}}\right).
\end{align}
\vspace{0.5cm}

\noindent
The continuity at $t = \tau^k_{\infty}$ can be obtained as follow. As $\tau^k_1 = \tau^k_\infty$ for $k \in \{1, \dots 2^{\eta d} - 1\}$, therefore we only need to consider the case $k = 2^{\eta d}$. As before, since 
$
\varrho_{i+1}(1, 2^\eta(\bs{x} - \bs{c}_k)) = \rho^{out}(1, 2^\eta(\bs{x} - \bs{c}_k)),
$
we have
\begin{align}
\mathcal{G}_i(\{\varrho_{i^\prime}\}_{i^\prime \in \mathbb{N}})(t, \cdot) \to  \sum_{1 \leq k^\prime \leq 2^{\eta d}} \rho^{out}(2^\eta(\bs{x} - \bs{c}_{k^\prime})) = \rho^{out}(\bs{x})\quad \text{in} \quad W^{-1, q} \quad \text{as} \quad t \nearrow \tau^k_{\infty},
\end{align}
which 
implies continuity at $t = \tau^{2^{\eta d}}_{\infty}$.


Collecting all the estimates above, we conclude
\begin{align}
 \frac{1}{2^{\gamma_3 i}}\norm{\mathcal{G}_i(\{\varrho_{i^\prime}\}_{i^\prime \in \mathbb{N}})}_{\text{Lip}_t W_{\bs{x}}^{-1, q}}
 \le 
 C(d,\eta,\beta,\nu,q,\gamma_3) 2^{\nu d i(1-\frac{1}{q}) - \gamma_3 i}(1 + \norm{\{\varrho_{i^\prime}\}_{i^\prime \in \mathbb{N}}}_{\ell^\infty_{\gamma_3}\text{Lip}_t W^{-1, q}_{\bs{x}}})\, ,
\end{align}
hence, $\mathcal{G}: \mathcal Z \to \mathcal Z$ if
\begin{equation}
    \frac{\nu d}{q'}  - \gamma_3 \le 0 \, .
\end{equation}


\subsection{Proof of Proposition~\ref{density field: main prop}(ii): $\mathcal G: \mathcal{Z}\to \mathcal{Z}$ is a contraction}

Fix $\{\varrho_{i^\prime}\}_{i^\prime \in \mathbb{N}} \, ,\{\varsigma_{i^\prime}\}_{i^\prime \in \mathbb{N}} \in \mathcal{Z}$. It is immediate to check the following estimates
\begin{align}
    &\| \mathcal{G}_i(\{\varrho_{i^\prime}\}_{i^\prime \in \mathbb{N}}) - \mathcal{G}_i(\{\varsigma_{i^\prime}\}_{i^\prime \in \mathbb{N}})
\|_{C^1(\mathcal{T}^k_{(1)};W^{-1,q})}
     \le \frac{2^{\nu d i}}{\tau^k_{mid} - \tau^k_1} \left(\frac{1}{2^{\eta + \nu i}}\right)^{1 + \frac{d}{q}} \norm{\varrho_1 - \varsigma_1}_{\text{Lip}_t W^{-1, q}_{\bs{x}}}.
     \\&
      \| \mathcal{G}_i(\{\varrho_{i^\prime}\}_{i^\prime \in \mathbb{N}}) - \mathcal{G}_i(\{\varsigma_{i^\prime}\}_{i^\prime \in \mathbb{N}})\|_{C^1(\mathcal{T}^k_{(2)};W^{-1,q})} = 0
      \\&
       \| \mathcal{G}_i(\{\varrho_{i^\prime}\}_{i^\prime \in \mathbb{N}}) - \mathcal{G}_i(\{\varsigma_{i^\prime}\}_{i^\prime \in \mathbb{N}})\|_{C^1(\mathcal{T}^k_{(3)};W^{-1,q})}
       \le
       \frac{1}{\tau^k_{\infty} - \tau^k_2} \left(\frac{1}{2^{\eta}}\right)^{1 + \frac{d}{q}} \norm{\varrho_{i+1} - \varsigma_{i+1}}_{\text{Lip}_t W^{-1, q}_{\bs{x}}} \, ,
\end{align}
leading to
\begin{align}
&\frac{1}{2^{\gamma_3 i}}\norm{\mathcal{G}_i(\{\varrho_{i^\prime}\}_{i^\prime \in \mathbb{N}})  - \mathcal{G}_i(\{\varsigma_{i^\prime}\}_{i^\prime \in \mathbb{N}})}_{ \text{Lip}_t W_{\bs{x}}^{-1, q}} \nonumber 
\\ & \quad \leq 
2^{\gamma_3}\cdot 2^{-\eta (1+\frac{d}{q})} \cdot
\sup_{i \in \mathbb{N}} \max\left\{\frac{2^{\nu d i(1-\frac{1}{q}) - \nu i - \gamma_3 i}}{\tau^k_{mid} - \tau^k_1} , \frac{1}{\tau^k_{\infty} - \tau^k_2} \right\} 
\norm{\{\varrho_{i^\prime}\}_{i^\prime \in \mathbb{N}} - \{\varsigma_{i^\prime}\}_{i^\prime \in \mathbb{N}}}_{\ell^\infty_{\gamma_3} \text{Lip}_t W^{-1, q}_{\bs{x}}}.
\end{align}
We choose $\gamma_3 = \frac{\nu d}{q'}$, so that the maximum is achieved at $i=1$. It is easy to see that if $\nu \ge 1-\log_2(2^\beta-1)$, then
\begin{equation}
   \sup_{i \in \mathbb{N}} \max\left\{\frac{2^{\nu d i(1-\frac{1}{q}) - \nu i - \gamma_3 i}}{\tau^k_{mid} - \tau^k_1} , \frac{1}{\tau^k_{\infty} - \tau^k_2} \right\}
   = 2^{\beta + \eta d} \, .
\end{equation}
Hence $\mathcal{G}$ is a contraction in $\mathcal Z$ provided
\begin{equation}
    2^{\frac{\nu d}{q'} - \eta(1+\frac{d}{q}) + \beta + \eta d} 
    =
    2^{\gamma_3 - \eta(1+\frac{d}{q}) + \beta + \eta d}
    < 1 \, ,
\end{equation}
which amounts to
$
q < \frac{\eta d + \nu d}{\eta d + \nu d - \eta + \beta} \, .
$

\section{Proof of Theorem \ref{Intro: main thm}}
Let $p,r$ be as in the statement of the theorem, and let us further assume without loss of generality that $\frac{1}{p} + \frac{1}{r} < 1 + \frac{1}{d} $, so that $\bs{u} \in C([0, 1]; L^{\frac{r}{r-1}})$ follows from 
$\bs{u} \in C([0, 1]; W^{1, p})$ by Sobolev embedding. Next, we choose parameter $\beta = 1/2$. We then choose two numbers $\ol{p}$ and $\ol{r}$ such that
\begin{align}
p < \ol{p}, \qquad r < \ol{r}, \qquad
1 < \frac{1}{\ol{p}} + \frac{d-1}{d \, \ol{r}} < \frac{1}{p} + \frac{d-1}{d \, r}.
\end{align}
Moreover, we ensure that $\ol{r}$ is large enough such that
\begin{align}
\label{an extra cond. on r}
3 \, \ol{r} \, \left(\frac{1}{\ol{p}} + \frac{d-1}{d \, \ol{r}} - 1\right) < \frac{\beta}{d}. 
\end{align}
Next, we define $\wt{\eta}$ and $\wt{\nu}$ as
\begin{align}
\ol{r} = \frac{\wt{\eta} + \wt{\nu}}{\wt{\nu}} \qquad \ol{p} = \frac{(\wt{\eta}+\wt{\nu})d}{\wt{\eta} d + \wt{\nu} + \beta}.
\end{align}
The condition (\ref{an extra cond. on r}) ensures that $\wt{\nu} > 3 > 1 + \beta - \log_2(2^\beta - 1)$. We finally define our choices of $\eta$ and $\nu$ as
\begin{align}
\eta = \left\lceil \, \wt{\eta} \, \right\rceil, \qquad \nu = \frac{\eta}{\ol{r} - 1} \geq \wt{\nu}.
\label{choice of eta and nu}
\end{align}
Here, $\left\lceil \, \cdot \, \right\rceil$ is the ceiling function. From (\ref{choice of eta and nu}), we see that 
$
r < \ol{r} = \frac{\eta + \nu}{\nu}.
$ 
As a result, we obtain
\begin{align}
p < \ol{p} = \frac{(\wt{\eta}+\wt{\nu})d}{\wt{\eta} d + \wt{\nu} + \beta} = \frac{\ol{r} \, d}{(\ol{r} - 1) d + 1 + \dfrac{\beta}{\wt{\nu}}} \leq \frac{\ol{r} \, d}{(\ol{r} - 1) d + 1 + \dfrac{\beta}{{\nu}}} = \frac{(\eta+\nu)d}{\eta d + \nu + \beta}.
\end{align}
Therefore, every element from the sequence $\{\bs{w}_i\}_{i \in \mathbb{N}}$ belongs to $C([0, 1]; W^{1, p})$. 

Next, we consider the space $\mathcal{Y}\cap \mathcal{Z}$ with the metric induced by 
\begin{equation}
    \| \{\varrho_i\}_{i\in \mathbb N} \|_{\mathcal{Y}\cap \mathcal{Z}} 
    := 
    \| \{\varrho_i\}_{i\in \mathbb N} \|_{\mathcal{Y}}
    +
     \| \{\varrho_i\}_{i\in \mathbb N} \|_{\mathcal{Z}} \, .
\end{equation}
If we choose 
\begin{equation}
    1 < q < \frac{(\eta + \nu)d}{(\eta + \nu)d + \beta - \eta} \, ,
\end{equation}
then $\mathcal{G}: \mathcal{Y}\cap \mathcal{Z} \to \mathcal{Y}\cap \mathcal{Z}$ is a contraction for suitable weights $\gamma_2$, $\gamma_3$, as a consequence of Proposition \ref{density field: main prop} and we let 
$\{\rho_i\}_{i\in \mathbb{N}}$ the unique fixed point of $\mathcal{G}$ in $\mathcal{Y}\cap \mathcal{Z}$. It is then clear from Proposition \ref{density field: main prop} that each $\rho_i$ belongs to $L^\infty([0, 1]; L^r(\mathbb{R}^d))$.

Next, we show that $\rho_i$ solves the continuity equation corresponding to the vector field $\bs{v}_i$, i.e.,
\begin{align}
\frac{d}{dt} \int_{\mathbb{R}^d} \rho_i \, \phi \, {\rm d} \bs{x} = - \int_{\mathbb{R}^d} \rho_i \, \bs{v}_i \cdot \nabla \phi \, {\rm d} \bs{x} \quad \text{for all} \quad \phi \in C^\infty_c(\mathbb{R}^d)
,
\label{density solves continuity: weak sense i}
\end{align}
for a.e. $t \in [0, 1]$ and for all $i \in \mathbb{N}$. For $i \in \mathbb{N}$, let $c_i$ be the measure of the subset of $[0,1]$ where  $\bs{v}_i$ and $\rho_i$ obey (\ref{density solves continuity: weak sense i}). 
From the definition of $\mathcal F$ and $\mathcal G$ 
and since $\bs{v}_i$ and $\rho_i$ represent a fixed point, we know that they solve (TE) in the intervals $\mathcal{T}^k_{(2)}$, which have measure $({2^\beta - 1}){2^{-\beta -1}}$, as well as in a fraction $c_1$ of $\mathcal{T}^k_{(1)}$ and in a fraction $c_{i+1}$ of $\mathcal{T}^k_{(3)}$. Hence
\begin{align}
& c_i \geq c_1 \frac{2^\beta - 1}{2^{\beta + 1}} + \frac{2^\beta - 1}{2^{\beta + 1}} + \frac{c_{i+1}}{2^{\beta + 1}}.
\end{align}
Taking the infimum over $i\in \mathbb N$ on both sides we get
\begin{align}
& \inf c_i \geq \inf c_i \frac{2^\beta - 1}{2^{\beta + 1}} + \frac{2^\beta - 1}{2^{\beta + 1}} + \frac{\inf c_i}{2^{\beta + 1}}. 
\end{align}
which implies that $\inf c_i \geq 1$.

Finally, we let the velocity field $\bs{u}(t, \cdot) = \bs{w}_1(1 - t, \cdot)$ and the density $\rho(t, \cdot) = \rho_1(1 - t, \cdot)$, which then satisfy all the requirements stated in Theorem \ref{Intro: main thm}.

\begin{figure}[h]
\centering
 \includegraphics[scale = 1.3]{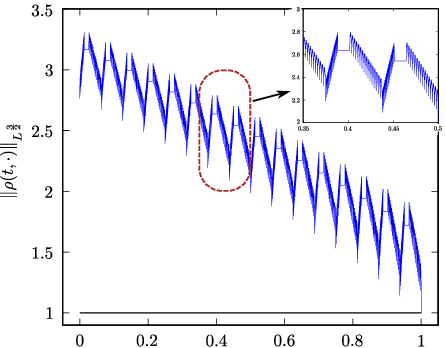}
 \caption{shows the plot of $ t\to \norm{\rho(t, \cdot)}_{L^{\frac{3}{2}}}$, where $\rho$ in this figure is the first element $\rho_1$ in $\{\rho_i\}_{i \in \mathbb{N}}$ from our  $L^r$ construction with $d = 2$, $\beta = 0.8$, $\nu = 2.3$, $\eta = 2$. Notice that $\norm{\rho(t, \cdot)}_{L^{\frac{3}{2}}}$ is bounded and continuous as opposed to our $L^1$ construction (see Figure~\ref{fig: plot Lr norm density L1 consturction}). 
 }
 \label{fig: org prb}
\end{figure}

\section{Building block}
\label{sec: building block}
In this section, we design the building blocks of our construction. The key function of our building block vector field $\ol{\bs{v}}_b(t, \bs{x})=\ol{\bs{v}}_b(t,\bs{x};\eta, a, s)$ is to ``spread'' cubes as depicted in figure \ref{fig: building block}. More precisely, given three parameters $\eta \in \mathbb{N}$ and $0 < 2s < a < 1$, we will build a divergence-free vector field $\ol{\bs{v}}_b(t,\bs{x};\eta, a, s)$ that moves the density field $\ol{\rho}_b(t,\bs{x};\eta, a, s)$, which is initially at $t=0$ concentrates on cubes of sizes $\frac{s}{2^\eta}$ placed uniformly in a grid fashion inside a cube of size $a$, to a final configuration at $t = 1$, where these cubes evenly spread (again in a grid fashion) inside a cube of size $1$.  Next, we write down the initial and final density field $\ol{\rho}^s$ and $\ol{\rho}^e$ in our building block construction.
\begin{align}
\ol{\rho}^s \coloneqq 
\begin{cases}
& 1 \quad \text{if} \quad \bs{x} \in \bigcup\limits_{k \in \{1, \dots 2^{\eta d}\}}\ol{Q}(a \bs{c}^\eta_k, \frac{s}{2^\eta}), \\
& 0 \quad \text{otherwise},
\end{cases}
\\
\ol{\rho}^e \coloneqq 
\begin{cases}
& 1 \quad \text{if} \quad \bs{x} \in \bigcup\limits_{k \in \{1, \dots 2^{\eta d}\}}\ol{Q}(\bs{c}^\eta_k, \frac{s}{2^\eta}), \\
& 0 \quad \text{otherwise}.
\end{cases}
\end{align}

\begin{figure}[h]
\centering
 \includegraphics[scale = 0.4]{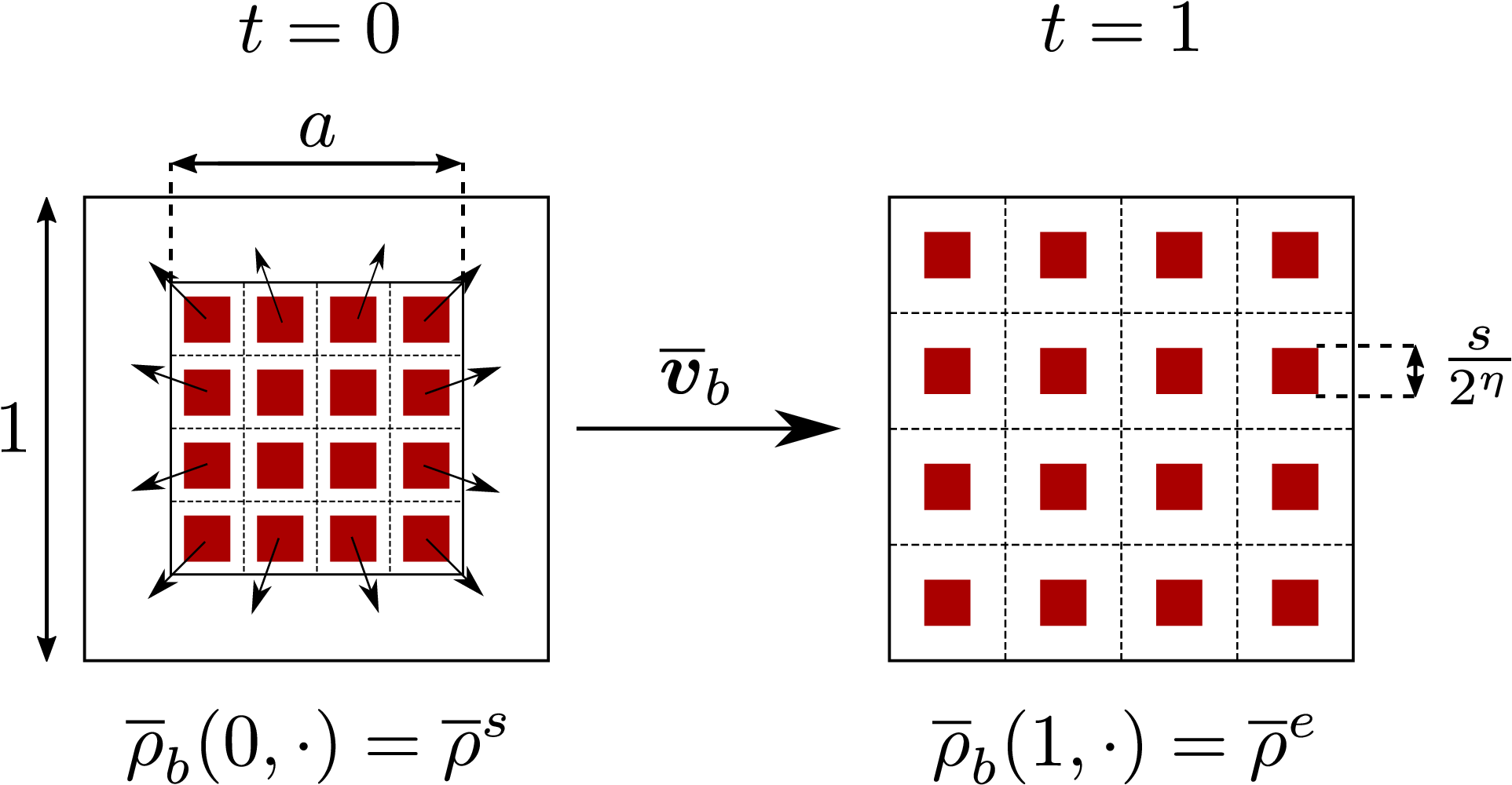}
 \caption{shows the action of the building block vector field $\ol{\bs{v}}_b$. At time $t = 0$ the density field is $\ol{\rho}^s$, which concentrates on $2^{\eta d}$ cubes each of size $\frac{s}{2^\eta}$ placed evenly inside another cube of size $a < 1$. The vector field $\ol{\bs{v}}_b$ spreads these cubes in $1$ amount of time such that in the final configuration these cubes are spaciously laid out inside a cube of size $1$.}
 \label{fig: building block}
\end{figure}



The main result of this section is the following proposition.
\begin{proposition}
\label{prop: building block vector field}
Given two parameters $0 < 2 s < a < 1$, there exists a divergence-free vector field $\ol{\bs{v}}_b \in C^{\infty}_c([0, 1] \times \mathbb{R}^d; \mathbb{R}^d)$, supported on $\left[\frac{1}{3}, \frac{2}{3}\right]\times\left[-\frac{1}{2}, \frac{1}{2}\right]^d$, and a solution  $\ol{\rho}_b \in L^\infty([0, 1] \times \mathbb{R}^d)$ of the continuity equation with vector field $\ol{\bs{v}}_b$  satisfying the following properties
\begin{equation}
    \label{eqn:supps}
    \mathscr{L}^d(\supp \ol{\bs{v}}_b(t, \cdot) ) \leq 2^d s^d \mbox{ for 0 }\leq t \leq 1,
\end{equation}
\begin{equation}
\label{eqn:reg}
\norm{\nabla_{(t, \bs{x})}^i \ol{\bs{v}}_b}_{L^\infty_{t, \bs{x}}} \leq
C(d,i) \left(\frac{2^\eta}{s} \right)^i\qquad \mbox{for every } i=0,1,...,
\end{equation}
$\ol{\rho}_b (t, \cdot)$ is either $0$ or $1$, the set $\{\ol{\rho}_b (t, \cdot)=1 \} \subseteq D$ has measure $s^d$ for every $t$ 
and satisfies
\begin{equation}
\label{eqn:rhose}
\ol{\rho}_b(t, \cdot) = 
\begin{cases}
\ol{\rho}_s \quad &\text{for} \quad 0 \leq t < \frac{1}{3}, \\
 \ol{\rho}_e \quad &\text{for} \quad \frac{2}{3} < t \leq 1.
\end{cases}
\end{equation}
\end{proposition}

The vector field $\ol{\bs{v_b}}$ is composed of $2^{\eta d }$ vector fields, each sending a cube with center $a \bs{c}^\eta_k$ and size $s/2^{\eta}$ to the corresponding cube (centered at $\bs{c}_k^\eta$) in the final configuration. The proof of above propostion is divided into two steps. In Step 1 of the proof we build a smooth, divergence-free and compactly supported vector field that translates a cube from one place at $t = 0$ to another place at $t = 1$ without deforming the cube. This kind of vector field, with minor technical adjustments, was introduced in \cite{kumar2023nonuniqueness}.
In Step 2, we assemble $2^{\eta d }$ copies of such a vector field and obtain $\ol{\bs{v}}_b$ with the desired properties.

\noindent
{\bf Step 1. A vector field that rigidly translates a cube of length $\lambda$} Let $0< \lambda$, $\bs{A}_0, \bs{A}_1\in \R^d $ and let  $\zeta : [0,1] \to  [0,1]$ be a smooth function such that $\zeta(x) = 0$ if $x \leq \frac{1}{3}$ and $\zeta(x) = 1$ if $x \geq \frac{2}{3}$. We define

\begin{eqnarray}
    \bs{A}_t \coloneqq \bs{A}_0 \left[1 - \zeta\left(t\right)\right]  + \bs{A}_1 \zeta\left(t\right). 
\end{eqnarray}
Next, we show that there is a smooth, divergence-free vector field $\wt{\bs{v}}(\cdot \, ; \bs{A}_0, \bs{A}_1, \lambda) : \mathbb{R} \times \mathbb{R}^d \to \mathbb{R}^d$ with the following properties:
 \begin{align}
\wt{\bs{v}}(t, \bs{x}) = (\bs{A}_1 - \bs{A}_0) \zeta^\prime\left(t\right) \qquad \mbox{for all }\bs{x} \in \ol{Q}(\bs{A}_t, \lambda),  
 \end{align}
\begin{align}
    \label{eqn:supp}
     \supp\wt{\bs{v}} \subseteq \left\{(t, \bs{x}): t\in \left[\frac{1}{3}, \frac{2}{3}\right], \; \bs{x}\in\ol{Q}(\bs{A}_t, 2\lambda)\right\},
\end{align} 
and
\begin{align}
\label{eqn:regtildev}\norm{\nabla_{(t, \bs{x})}^i\wt{\bs{v}}}_{L^\infty_{t, \bs{x}}} \lesssim \frac{|\bs{A}_0 - \bs{A}_1|}{\lambda^i}  \quad \forall i=0,1, \dots .
\end{align}
Furthermore, there is a density field $\wt{\rho}(\cdot, \bs{A}_0, \bs{A}_1, \lambda): \mathbb{R}^d \to \mathbb{R}$ defined as
 \begin{align}
    \label{eqn:itsolves}
   \wt{\rho}(t,x ; \bs{A}_0, \bs{A}_1, \lambda) = 1_{Q(\bs{A}_t, \lambda)},
\end{align}
which solves the transport equation with vector field $\wt{\bs{v}}$.

The construction of $\wt{\bs{v}}$ and $\wt{\rho}$ proceeds as follows. Let $\bs{\xi} \in \R^d$ with $|\bs{\xi}|=1$ and let $\bs{\xi}_1,...,\bs{\xi}_{d-1} \in \R^d$ be vectors such that along with $\bs{\xi}$ they form  an orthonormal basis of $\R^d$. Let $\varphi \in C^\infty_c(\left[-\frac{5}{4},\frac{5}{4}\right])$ be a smooth cutoff function, which equals $1$ in $[-1,1]$.
We define a vector field $\bs{b}_2 : \mathbb{R}^2 \to \mathbb{R}^2$ in dimension $d = 2$ as 
$$\bs{b}_2(x_1,x_2;\bs{\xi}) =- \nabla^\perp [ \bs{\xi}^\perp  \cdot (x_1,x_2) \; \varphi(x_1) \varphi(x_2)]$$
where $\bs{\xi}^\perp = (\xi_2, -\xi_1)$ and $\nabla^\perp = (\partial_2, -\partial_1).$
In dimensions $d \geq 3$, we define the vector field $\bs{b}_d: \mathbb{R}^d \to \mathbb{R}^d$ as
$$\bs{b}_d(\bs{x};\bs{\xi}) = \bs{b}_2(\bs{x}\cdot \bs{\xi}, \bs{x} \cdot \bs{\xi}_1; \bs{\xi})  \varphi(\bs{x} \cdot \bs{\xi}_2) ...\varphi(\bs{x} \cdot \bs{\xi}_{d-1}).$$
The velocity field $\bs{b}_d$ for $d \geq 2$ is smooth, divergence-free, compactly supported in $Q(\bs{0}, 2)$ and have the following properties:
\begin{equation}
    \label{eqn:bincube}
\bs{b}_d(\bs{x};\bs{\xi}) = \bs{\xi}\qquad \mbox{for every }\bs{x} \in \ol{Q}(\bs{0}, 1),
\end{equation} 
$$\norm{\nabla^i \bs{b}_d}_{L^\infty} \lesssim1 \quad \forall \; i =0,1, \dots .$$


Now we rescale and translate this vector field in a time dependent fashion and define a vector field $\wt{\bs{v}}: \mathbb{R}^d \to \mathbb{R}^d$ as
\begin{eqnarray}
    \wt{\bs{v}}(t, \bs{x}; \bs{A}_0, \bs{A}_1, \lambda) \coloneqq |\bs{A}_t^\prime| \bs{b}_d\left(\frac{\bs{x} - \bs{A}_t}{\lambda} ; \frac{\bs{A}_1 - \bs{A}_0}{|\bs{A}_1 - \bs{A}_0|} \right).
    \label{eqn: v tilde}
\end{eqnarray}
Notice that, for any $\bs{x} \in \ol{Q}(\bs{A}_0, \lambda)$, the curve $\bs{x} - \bs{A}_0 + \bs{A}_t$ is an integral curve of the velocity field \eqref{eqn: v tilde}. Therefore, by the representation of solutions of the transport equation with integral curves, this implies \eqref{eqn:itsolves} solves
\begin{align}
\partial_t \wt{\rho} + \wt{\bs{v}} \cdot \nabla \wt{\rho} = 0.
\label{eqn: rho tilde solves continuity}
\end{align}



\begin{figure}[h]
\centering
 \includegraphics[scale = 0.3]{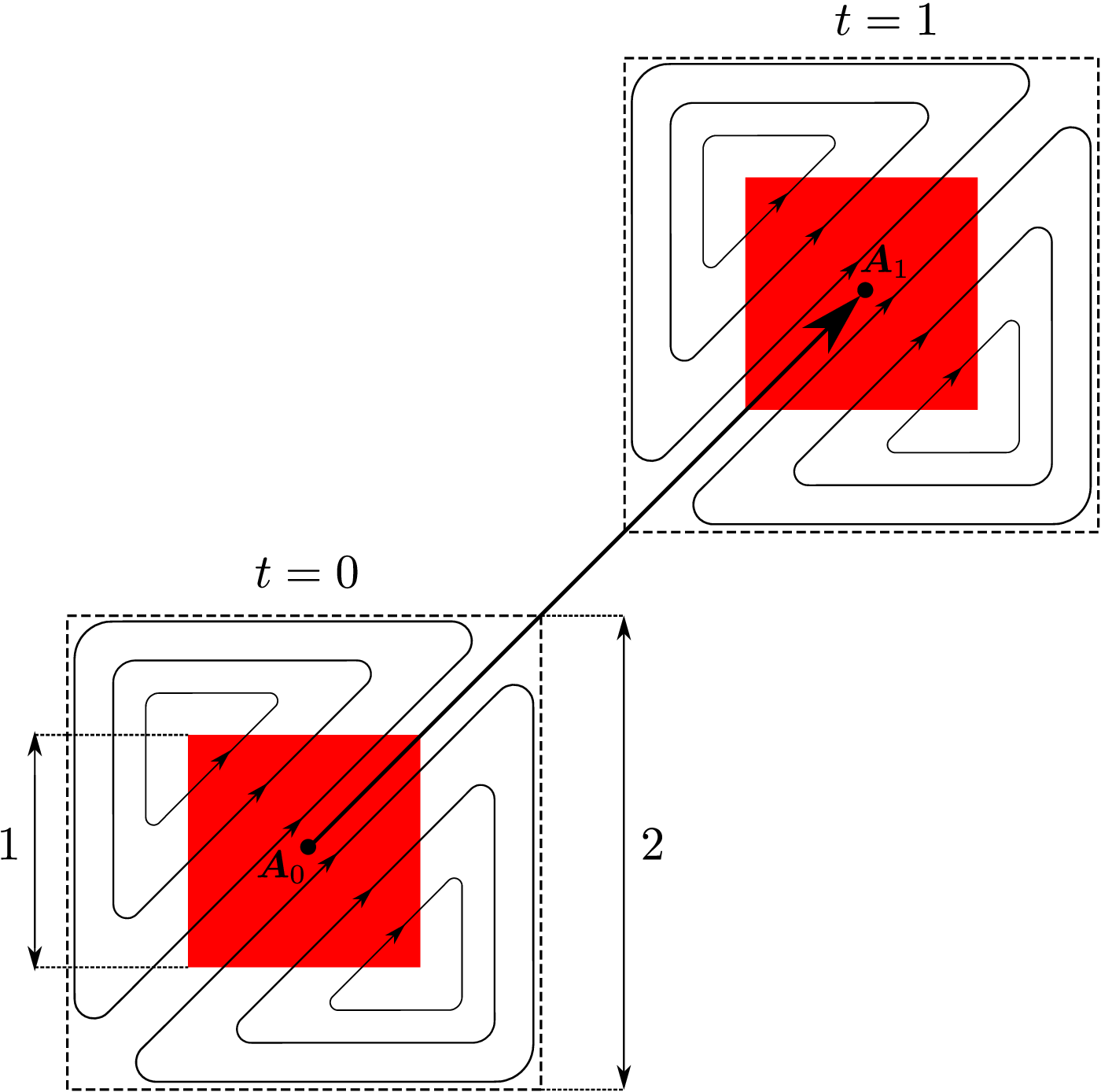}
 \caption{depicts the building block vector field which rigidly translates a cube of size $1$ centered at $\bs{A}_0$ at time $t = 0$ to a cube centered at $\bs{A}_1$ at time $t=1$. The building block vector field is divergence free and it is uniform in the support of the cube as it moves through space. Moreover, the vector field stays supported in double the size of cube it translates.}
 \label{fig: blob flow}
\end{figure}

\noindent
{\bf Step 2. Assembling several copies of $\wt{\bs{v}}$ and $\wt{\rho}$}
Given $k \in \{1, \dots 2^{\eta d}\}$, we first define a vector field $\ol{\bs{v}}_k$ and a density field $\ol{\rho}_k$ as
\begin{subequations}
\begin{align}
\ol{\bs{v}}_k(t, \bs{x}; \eta, a, s) \coloneqq \wt{\bs{v}}\left(t, \bs{x}; a\bs{c}_k^\eta, \bs{c}_k^\eta, \frac{s}{2^\eta}\right) \\
\ol{\rho}_k (t, \bs{x}; \eta, a, s) \coloneqq \wt{\rho}\left(t, \bs{x}; a\bs{c}_k^\eta, \bs{c}_k^\eta, \frac{s}{2^\eta}\right).
\end{align}
\end{subequations}
From \eqref{eqn:supp}, we know that 
for $t \in [0, 1]$
\begin{subequations}
\begin{align}\label{eqn:boh}
& \supp_{\bs{x}} \tilde \rho_k \subseteq\supp_{\bs{x}} \ol{\bs{v}}_k(t, \cdot) \subseteq \ol{Q}\left(\bs{x}_p^k(t), \frac{2s}{2^\eta}\right), 
\end{align}
\end{subequations}
where 
\begin{subequations}
\begin{align}
& \bs{x}^k_p(t) =  (1 - \zeta(t))a \bs{c}_k^\eta +  \zeta(t) \bs{c}_k^\eta
. 
\end{align}
\end{subequations}
We next observe that the cubes in \eqref{eqn:boh} are disjoint for different $k$. Indeed for $k \neq k^\prime$, we have 
\begin{align}
|\bs{x}^k_p(t) - \bs{x}^{k^\prime}_p(t)| \geq a|\bs{c}_k^\eta - \bs{c}_{k^\prime}^\eta| \geq \frac{a}{2^\eta} >\frac{2s}{2^\eta}.
\end{align}
Hence, the building block vector field $\ol{\bs{v}}_b$ and density field $\ol{\rho}_b$ defined as
\begin{subequations}
\begin{align}
\ol{\bs{v}}_b \coloneqq \sum_{k = 1}^{2^{\eta d}} \ol{\bs{v}}_k, \nonumber \\
\ol{\rho}_b \coloneqq \sum_{k = 1}^{2^{\eta d}} \ol{\rho}_k
\end{align}
\end{subequations}
are composed of addends with disjoint support contained in the union of $2^{\eta d}$ cubes of size $2s/2^\eta$, hence obtaining \eqref{eqn:supps}. Moreover $\bar \rho_b$ solves the continuity equation in view of \eqref{eqn: rho tilde solves continuity} with prescribed initial and final conditions given in \eqref{eqn:rhose} and the velocity field $\ol{\bs{v}}_b$ satisfies the regularity estimates 
\eqref{eqn:reg}, which can be seen from \eqref{eqn:regtildev}.







\bibliographystyle{halpha-abbrv}
\bibliography{references.bib}

\end{document}